\RequirePackage[l2tabu, orthodox]{nag}
\documentclass[a4paper,noamsfonts]{amsart}

\usepackage[utf8]{inputenc}
\usepackage[T1]{fontenc}

\usepackage{cfr-lm}
\usepackage{bm}
\newcommand{\mathbold}{\bm}

\usepackage{amsmath}
\usepackage{amsthm}
\usepackage{amssymb}
\usepackage[foot]{amsaddr}
\usepackage{mathtools}

\usepackage{tikz}
\usetikzlibrary{cd}
\usepackage{framed}
\usepackage{xcolor}
\usepackage{graphicx}
\usepackage{pinlabel}
\usepackage{subcaption}
\usepackage{enumitem}
\usepackage[backend=bibtex,style=numeric, maxnames=8, isbn=false, sortcites=true, backref]{biblatex}
\usepackage{hyperref}
\hypersetup{
    colorlinks,
    linkcolor={red!50!black},
    citecolor={blue!50!black},
    urlcolor={blue!80!black}
}
\usepackage[capitalise,noabbrev]{cleveref} % load after hyperref
\usepackage{microtype}

\newtheorem{thm}{Theorem}
\numberwithin{thm}{section}
\newtheorem*{thm*}{Theorem}
\newtheorem{mainthm}{Main Theorem}

\newtheorem{lem}[thm]{Lemma}
\AddToHook{env/lem/begin}{\crefalias{thm}{lem}}
\newtheorem{prp}[thm]{Proposition}
\AddToHook{env/prp/begin}{\crefalias{thm}{prp}}
\newtheorem{cor}[thm]{Corollary}
\AddToHook{env/cor/begin}{\crefalias{thm}{cor}}
\newtheorem{obs}[thm]{Observation}
\AddToHook{env/obs/begin}{\crefalias{thm}{obs}}
\theoremstyle{definition}
\newtheorem{defn}[thm]{Definition}
\AddToHook{env/defn/begin}{\crefalias{thm}{defn}}
\newtheorem{ex}[thm]{Example}
\AddToHook{env/ex/begin}{\crefalias{thm}{ex}}
\newtheorem*{discprob*}{Geometric Recognition Problem}
\newtheorem{cons}[thm]{Construction}
\AddToHook{env/cons/begin}{\crefalias{thm}{cons}}
\theoremstyle{remark}
\newtheorem{rem}[thm]{Remark}
\newtheorem{history}[thm]{Historical remark}

\newtheorem{subcons}{Step}[thm]

\makeatletter
\newcommand{\WronglyDeclarePairedDelimiter}[3]{%
  \expandafter\DeclarePairedDelimiter\csname RIGHT\string#1\endcsname{#2}{#3}%\lstinputlisting[language=Octave]{BitXorMatrix.m}
  \newcommand#1{%
    \@ifstar{\csname RIGHT\string#1\endcsname}
            {\csname RIGHT\string#1\endcsname*}%
  }%
}
\makeatother
\WronglyDeclarePairedDelimiter{\abs}{\lvert}{\rvert}
\WronglyDeclarePairedDelimiter{\norm}{\lVert}{\rVert}

\newcommand{\df}{\textit}

\newcommand{\union}{\cup}
\newcommand{\inter}{\cap}

\newcommand{\mc}{\mathcal}

\usepackage{scalefnt}
\newcommand\ScaleExists[1]{\vcenter{\hbox{\scalefont{#1}$\exists$}}}
\DeclareMathOperator*\bigexists{%
  \vphantom\sum
  \mathchoice{\ScaleExists{2.5}}{\ScaleExists{1.4}}{\ScaleExists{1}}{\ScaleExists{0.75}}}

\newcommand{\git}{\mathord{
  \mathchoice{/\mkern-4mu/}% \displaystyle
    {/\mkern-4mu/}% \textstyle
    {/\mkern-3mu/}% \scriptstyle
    {/\mkern-3mu/}}}% \scriptscriptstyle

\renewcommand{\H}{\mathbb{H}}
\renewcommand{\P}{\mathbb{P}}
\newcommand{\R}{\mathbb{R}}
\newcommand{\C}{\mathbb{C}}

\newcommand{\Z}{\mathbb{Z}}
\newcommand{\Q}{\mathbb{Q}}

\newcommand{\Sph}{\mathbb{S}}
\newcommand{\V}{\mathbold{Z}}
\newcommand{\FPer}{\tilde{\Gamma}}

\DeclareMathOperator{\im}{im}
\DeclareMathOperator{\PSL}{\mathsf{PSL}}

\DeclareMathOperator{\SL}{\mathsf{SL}}

\DeclareMathOperator{\PO}{\mathsf{PO}}

\DeclareMathOperator{\Isom}{Isom}
\DeclareMathOperator{\Stab}{Stab}

\DeclareMathOperator{\QH}{QH}

\DeclareMathOperator{\Fix}{Fix}
\DeclareMathOperator{\tr}{tr}
\DeclareMathOperator{\Hom}{Hom}

\DeclareMathOperator{\rad}{rad}
\DeclareMathOperator{\cen}{cen}
\DeclareMathOperator{\Pen}{Pen}
\DeclareMathOperator{\Ax}{Ax}

\DeclareMathOperator{\hconv}{h.\!conv}

\begin{document}

\title{Peripheral subgroups of Kleinian groups}
\author[A. Elzenaar]{Alex Elzenaar}
\address{School of Mathematics, Monash University, Melbourne, Australia}
\email{alexander.elzenaar@monash.edu}
\thanks{The author was supported by an Australian Government Research Training Program (RTP) Scholarship during the period that this work was undertaken.}


\begin{abstract}
  The conformal boundary of a hyperbolic $3$-manifold $M$ is a union of Riemann surfaces. If any of these Riemann surfaces has a nontrivial
  Teichm\"uller space, then the hyperbolic metric of $M$ can be deformed quasi-isometrically. These deformations correspond to small pertubations
  in the matrices of the holonomy group $ \pi_1(M) \subset \PSL(2,\C) $, which together give an island of discrete representations around the identity map in $ X=\Hom(\pi_1(M), \PSL(2,\C)) $.
  Determining the extent of this island is a hard problem. If $M$ is geometrically finite and its convex core boundary is pleated only along simple closed curves,
  then we cut up its conformal boundary in a way governed by the pleating combinatorics to produce a fundamental domain for $ \pi_1(M) $ that
  is combinatorially stable under small deformations, even those which change the pleating structure. We give a computable region in $X$, cut out by polynomial
  inequalities over $\R$, within which this fundamental domain is valid: all the groups in the region have peripheral structures that
  look `coarsely similar', in that they come from real-algebraically deforming a fixed conformal polygon and its side-pairings. The union of all these
  regions for different pleating laminations gives a countable cover, with sets of controlled topology, of the entire quasi-isometric deformation space
  of $ \pi_1(M) $---which is known to be topologically wild.
\end{abstract}

\maketitle

\section{Introduction}
A discrete subgroup of $ \PSL(2,\C) $ is the holonomy group of a complete hyperbolic $3$-orbifold and is called a
\df{Kleinian group}. These groups were first introduced by Poincar\'e~\cite{poincare83} and later studied extensively from the point of view
of quasiconformal analysis and geometric function theory. They experienced a resurgence of interest following their application in Thurston's
revolutionary work in three-dimensional topology~\cite{thurston82}, and over the past two decades many structural conjectures have been settled:
the tameness conjecture by Agol~\cite{agol04} and Calegari and Gabai~\cite{calegari06}, the ending lamination conjecture by Brock, Canary, and
Minsky~\cite{minsky10,brock12}, and the density conjecture by Namazi and Souto~\cite{namazi12} and Ohshika~\cite{ohshika11}. Important questions
still remain about concrete problems, as the invariants and properties that these conjectures describe are often hard to translate into a
computable or otherwise explicit algebraic form when a particular family of Kleinian groups is written down.

We fix a geometrically finite hyperbolic $3$-manifold $ M $ (roughly, a $3$-manifold that is obtained by taking a finite sided
polyhedron in $ \H^3 $ and identifying pairs of faces) with non-empty conformal boundary, and we study the locus of discrete $\PSL(2,\C)$-representations
of the holonomy group $ \pi_1(M) $. More precisely, we determine---computably---the extent of the component of the discreteness locus which contains $ \pi_1(M) $,
by cutting it up into simple pieces so that the representations within each piece have a combinatorially stable action on $ \partial\H^3 $; i.e.\ they have
combinatorially stable fundamental domains. We use concrete hyperbolic geometry methods in order to construct appropriate fundamental domains, which depend
algebraically on the coefficients of a generating set and which are closely related to Ford domains but take into account the hyperbolic
geometry of the peripheral structures of the group. The deformations of the fundamental domains as the group parameters change are controlled
using techniques from real algebraic geometry and M\"obius geometry. Iterating through all the semi-algebraic regions produced gives a discreteness
and group recognition semi-algorithm for the representations of $\pi_1(M) $.

\subsection{The motivating problem}
To illustrate the difficulty of moving between the geometry of $M$ and the algebra of $\pi_1(M)$, we consider two natural examples of parameter spaces that
have arisen in previous work in geometric topology.
\begin{ex}\label{ex:rileychar}
  The space of Kleinian groups uniformising a $3$-ball with two drilled ideal arcs (represented by rank $1$ parabolics), the \df{Riley slice}~\cite{keen94,ems21}, is a subset
  of the set of groups generated by the parameterised pair of matrices
  \begin{displaymath}
    \C \ni \rho \mapsto \left(
    \begin{bmatrix} 1 & 1 \\ 0 & 1 \end{bmatrix}, \begin{bmatrix} 1 & 0 \\ \rho & 1 \end{bmatrix} \right) \in \PSL(2,\C)^2.
  \end{displaymath}
\end{ex}
\begin{ex}\label{ex:12cbchar}
  The space of Kleinian groups which uniformise $ (1;2)$-compression bodies~\cite{lackenbypurcell13,elzenaar24c} is a subset of the set of groups generated by the parameterised
  triple of matrices
  \begin{displaymath}
    \C^3 \ni (\alpha,\beta,\lambda) \mapsto \left(
    \begin{bmatrix} 1 & \alpha \\ 0 & 1 \end{bmatrix}, \begin{bmatrix} 1 & \beta \\ 0 & 1 \end{bmatrix}, \begin{bmatrix} \lambda & \lambda^2 - 1 \\ 1 & \lambda \end{bmatrix} \right) \in \PSL(2,\C)^3.
  \end{displaymath}
\end{ex}
The spaces parameterised by the maps of \cref{ex:rileychar} and \cref{ex:12cbchar} are very convenient to work with from the point of view of algebra: the maps are low-degree and defined over $ \Q $,
and the image groups are all non-conjugate so they define parameterisations of sections of the character variety. However, from the point of view of geometry, these spaces are not very well-behaved:
the discreteness locus in each is a complicated fractal subset which is very hard to compute. In the case of \cref{ex:rileychar}, the estimation of bounds on the discreteness locus has occupied
authors from as early as the 1950s to the present day, see the discussion in~\cite{ems24bd}. In general the locus of discreteness within a $\PSL(2,\C)$-character variety is a highly non-algebraic
subset that is known to have topologically wild closure~\cite{bromberg11,ohshika20}. Attempting to study the relationship between algebra and geometry in examples like these leads to the motivating problem of the paper.
\begin{discprob*}\label{discprob}
  Let $ G $ be a Kleinian group which admits deformations, and let $ t \mapsto G_t $ be an algebraic parameterisation of a subset of the character variety $ X(G) = \Hom(G,\PSL(2,\C))\git\PSL(2,\C) $,
  where $ G_0 = G $. The problem is, when given a parameter value $ t \in \C^n $, to decide whether $ \H^3/G_t $ is quasi-isometric to $ \H^3/G_0 $. We can rephrase this in view of
  quasiconformal deformation theory: let $ U $ be the subset of $ \C^n $ consisting of all parameters $ s $ such that $ G_s $ is discrete, and let $ U^0 $
  be the connected component of $ \operatorname{int} U $ that contains $0$; then the problem is to determine whether $ t $ lies in $ U^0 $.
\end{discprob*}

Our strategy for attacking this problem is motivated by a study of earlier work by Keen and Series~\cite{keen93,keen94} and previous work of the author with Martin and Schillewaert~\cite{ems21}. We first
construct a dense `lamination', by semi-algebraic\footnote{Throughout this paper, \df{semi-algebraic} is shorthand for semi-algebraic over $ \R $. Roughly speaking this means
sets defined by finitely many polynomial inequalities in $ \R^n $, see Basu, Pollack, and Roy~\cite[\S 2.3]{basu}.} sets, of the quasiconformal deformation space
of $G$. Each `leaf' of this lamination consists of groups with specified convex core bending lamination. We then thicken the leaves semi-algebraically; since groups with rational bending laminations
are dense, we obtain a semi-algebraic covering of the entire space.

We begin in \cref{sec:locglo} with the introduction of definitions and notation that will be used throughout the paper, together with some general machinery for decomposing
geometrically finite Kleinian groups in a purely formal way. Following this, in \cref{sec:peripheral}, we introduce the special families of subgroups (called \df{circle chains})
which induce decompositions of geometrically finite Kleinian groups in ways that are compatible with the geometry of the convex core quotient. The semi-algebraic leaves of the lamination (called
\df{pleating varieties}) are defined in \cref{sec:varieties} in terms of the existence of these geometric decompositions. Following this, in \cref{sec:formaldomains} we give
the algebraic structure that is used to talk about parameterisations of fundamental domains. Then, in \cref{sec:fundom}, we state and prove the two main theorems of the paper
which we will now describe.

\subsection{The two main theorems}
Our first main result, in \cref{sec:fundom_global}, is a construction of canonical fundamental domains for groups on each pleating variety. These domains can be
interpreted as generalisations of Ford domains~\cites[\S 19]{ford}[\S II.H]{maskit}, a class of fundamental domains which has been extensively studied in a number of different
contexts including in early work of Riley on hyperbolic knot theory~\cite{riley82}, the theory of punctured torus groups~\cite{akiyoshi,jorgensen03} and later punctured Klein
bottle groups~\cite{furokawa15}, compression body groups~\cite{lackenbypurcell13}, arithmetic Fuchsian groups~\cite{johansson00}, and discreteness testing for closed $3$-manifold
groups~\cite{manning02}. The appeal of Ford domains is due to their simplicity of definition, as well as their remarkable connections to other constructions in hyperbolic geometry, for instance being dual to
canonical triangulations~\cite[\S 4]{epstein88} which are themselves of extensive and continuing interest in hyperbolic knot theory. The advantage of our construction is that it
conforms to the shape of the peripheral structures of the groups being studied and so is more stable under deformation than the standard Ford domain. This conformality
comes from defining the domains in terms of hyperbolic geometry intrinsic to the groups; by comparison, the Euclidean geometry used to define Ford domains requires the
choice\ of a `good' conjugacy representative of $ G $ in order to get useful results. We give a rough formulation of the result here.

\begin{mainthm}[{\cref{thm:funddom}}]\label{mainthm1}
  Let $G$ be a geometrically finite torsion-free Kleinian group, and let $P \subset X(G) $ be the set of all geometrically finite Kleinian groups with the same convex core bending
  lamination\footnote{All laminations are purely topological unless stated otherwise and do not come equipped with measures.} as $G$. There exists an explicit family of polynomial
  inequalities on $ \C $, with coefficients varying $\R$-algebraically with the matrix entries of the generators of $ G $, such that for all $ \tilde{G} \in P $ the feasible region
  of the inequalities gives a fundamental domain for the action of $ \tilde{G} $ on its domain of discontinuity. In addition, there is a countable semi-algebraic subdivison of $ P $
  such that, within each region, the side-pairing maps for the domain \emph{also} vary algebraically.
\end{mainthm}

This theorem statement is rather inexplicit in the sense that we have not written it in a way which gives the inequalities on the nose. However, the proof gives a recipe for writing
down the inequalities in question in terms of any desired parameterisation of $ X(G) $. This means that if one is interested in working with any one particular family of groups then the inequalities
can be made explicit, the only real difficulty being computations of trace polynomials for the elements of $ G $ representing the convex core bending lamination, which is a hard problem to study abstractly (usually
specific features of the groups under study must be used to make the computations tractable). Geometrically, for each flat piece of the convex core of $ G $ one obtains a small family of inequalities
that encode a fundamental domain for the Fuchsian group uniformising that flat piece. Most of the technical difficulty comes in ensuring that everything depends algebraically on
the starting parameterisation of $ X(G) $ in a controlled way; the actual fundamental domains constructed are cut out by circles in a combinatorially straightforward way.

\begin{ex}\label{ex:punc_sph_1}
  As an example where our results can be applied relatively simply, consider a genus two Schottky group $ G $, i.e. a purely loxodromic
  Kleinian group which is free of rank $2$. Then $ \H^3/G $ has a single conformal end, a compact surface of genus $2$. If the bending lamination of $ \H^3/G $ at this end is a rational
  lamination $ \Lambda$, then there are exactly two Fuchsian groups involved $\Lambda$ has two complementary regions. Suppose we pinch two of the leaves of $\Lambda$ to parabolics, so we are
  in the setting of \cref{ex:rileychar} and the representation space has a single complex parameter $ \rho $. The semi-algebraic subdivision of $P$ mentioned in the theorem statement
  consists of only one region in the complex plane (the whole of $P$), and $P$ is cut out by two real polynomial inequalities on $\rho$: one trace inequality on
  the word of $G$ representing the non-parabolic leaf of $\Lambda$ and one inequality picking out the correct connected component. The proof of the theorem furnishes us with a semi-algebraic
  map from a well-behaved subset of $ \Hom(F_2, \PSL(2,\C)) $ to the set of $10$-tuples of circles (six circles for each Fuchsian group, with two shared between them) in the Riemann
  sphere so that one of the components of the common exterior of those circles is a fundamental domain for the action of each representation on the Riemann sphere, with side-pairings
  coming from the same words in every representation. In this special case the construction can be viewed morally as a generalisation of our earlier work with Martin and Schillewaert~\cite{ems21},
  but our results show that a similar construction works for all hyperbolic $3$-manifolds with exactly one conformal end where that end is a $k$-punctured sphere for $ k \geq 3 $;
  see \cref{ex:punc_sph_2}.
\end{ex}

In \cref{sec:twisting}, we extend this theory to give fundamental domains on thickenings of the strands. More precisely, we take the algebraic functions which
assign a fundamental domain to each group on a pleating ray and we extend them to algebraic functions on the whole parameter space which assign to groups `putative fundamental
domains' (i.e.\ formal conformal polygons with side-pairing relations). We then find a semi-algebraic subset $Y$ of $ X(G) $ within which these putative fundamental domains are actually
topological polygons without any degeneracies like overlapping sides, so $Y$ is a set on which this extended function actually produces fundamental domains:
\begin{mainthm}[{\cref{thm:mainthm}}]\label{mainthm2}
  Let $G$ be a geometrically finite torsion-free Kleinian group. Then there exists a full-dimensional semi-algebraic subset $ Y $ of the quasiconformal
  deformation space $ \QH(G) $ which contains the locus $P$ of \cref{mainthm1} and such that the assignment of fundamental domains to groups in $P$ extends to give algebraically varying
  fundamental domains for all groups in $ Y $.
\end{mainthm}
Again, although the statement we have given here is inexplicit, the proof follows an entirely explicit construction based on the combinatorics and geometry of the
flat pieces of the convex core of $ \H^3/G $. The idea is to take the unique extension of the $\R$-algebraic map of \cref{mainthm1} from $ P $ to $ Y $;
then the only problem is that the circles defining fundamental domains for the flat pieces of $G$ will now twist around since we
are leaving the locus where the corresponding groups are Fuchsian, and so we need to be a little clever to define the deformation of the pieces of fundamental domain
joining them together. The most technical part of the paper is \cref{cons:twisting}, where we deal with this problem by explicitly giving a deformation of these `joining pieces'
so that they also twist around algebraically with the parameterisation of the generators of $ G$.

One thing we obtain from \cref{mainthm2} is a decomposition of $ \QH(G) $ into semi-algebraic pieces so that the representations within each piece have combinatorially stable
fundamental domains. Constructions with similar goals have been carried out in the past: classically, constructions of decompositions of deformation spaces into pieces within
which groups have stable Ford domains were been carried out by J\o{}rgensen~\cite{jorgensen03} for once-punctured torus reresentations (see also Akiyoshi, Sakuma, Wada, and Yamashita~\cite{akiyoshi}),
and by Riley for four-punctured sphere representations (unpublished, see~\cite[Figure~0.2a]{akiyoshi}), and these constructions are essentially theories which give
hyperbolic triangulations of the corresponding manifolds. The modern descendent of this sort of construction is the model manifold theory originated by Minsky~\cite{minsky99,minsky10}
for the purpose of studying ending laminations, but in general the combinatorial decomposition is of a bi-Lipschitz `model' of the manifold, and a fundamental domain is not found
explicitly in $ \H^3$. Unlike these constructions, we are interested only in fundamental domains on the Riemann sphere, since this is enough to detect membership of the quasiconformal
deformation space. It is this simplification of the problem which makes it tractable.

\begin{rem}
  In \cref{cor:h3map} we describe a Voronoi construction which produces an $\H^3$-fundamental domain from our domains, but we do not have any control over the nature of this three-dimensional domain.
  We suspect that if one is willing to lose the conditions that everything is algebraic in the generators then it may be possible to give enough control over this Voronoi construction
  to make explicit its relation to the model manifold theory. Very roughly, the latter depends on decompositions according to short geodesics controlled by the complexes of curves of the surface
  ends of the manifold and in the large scale this is the same kind of data that we are detecting with our Voronoi construction. Our $\H^3$-domains can be viewed as analogous to the canonical
  triangulations constructed by Epstein and Penner~\cite{epstein88} for cofinite volume Kleinian groups, and these are known in special cases to be related to the model manifold theory---as one
  example, the Keen--Series theory of peripheral groups for the Riley slice has been explicitly connected to model manifold theory by Ohshika and Miyachi~\cite{ohshika10}. Further development in this direction seems interesting from
  the point of view of making the ending lamination theorem more practical to work with computationally, but we will say no more about this in the current paper.
\end{rem}

Apart from its intrinsic interest in giving information about the conformal action, \cref{mainthm2} has an application to the Geometric Recognition Problem: it produces a discreteness certificate for elements in $ Y $,
and hence exhibits the semi-algebraic set $Y$ as a full-dimensional well-behaved set of groups that lie in the same component of the discreteness locus as $G$. In fact, it follows from the constructions
themselves that there is a countable sequence of semi-algebraic sets (indexed by rational laminations on the conformal boundary of $ \H^3/G $) which exhaust the full component of the
discreteness locus containing $ G $. Hence in principle we obtain an semi-algorithm for certifying membership of representations $ G \to \PSL(2,\C) $ in $ \QH(G) $ and, as a consequence,
certifying discreteness of their images.

\subsection{Relation to the discreteness problem}
The problem of determining the discreteness of a group $ G \leq \PSL(2,\C) $ given in terms of matrix generators is an important and long-standing open problem in geometric
group theory. The Geometric Recognition Problem is slightly weaker than the full discreteness problem for $ \PSL(2,\C) $, because instead of asking for a certificate
of indiscreteness we ask for \emph{either} a certificate of indiscreteness, \emph{or} a certificate that the quotient $3$-manifold is incorrect. Putting it another way, we are
asking for a solution of the discreteness problem in the presence of some guess about the structure of $G$, and we are allowed to fail without deciding discreteness if we
verify that the guess about $G$ is incorrect. In practice the obstructions to solving the discreteness problem are the same as the obstructions to solving the Geometric Recognition Problem: essentially,
the existence of finitely generated groups without a finite-sided fundamental polyhedron.

The difficulty of the discreteness problem for $ \PSL(2,\C) $ is well-studied. As one example, Kapovich~\cite{kapovich16} has shown that discreteness in $ \PSL(2,\C) $ is undecidable in the BSS model of
computation, a model which includes oracles for computation over $ \R $ and which is sufficient to do computations for real semi-algebraic sets. In general, the state
of the art for $ \PSL(2,\C) $ seems quite far away from practical algorithms or semi-algorithms for discreteness testing: the only tests available are tests that can
certify discreteness or indiscreteness in special cases, or tests for discreteness together with some additional property. We refer the reader to recent survey
articles of Kapovich~\cite{kapovich23discreteness} and Gilman~\cite{gilman23} for further discussion and references.

While our methods are impractical in practice (for example, the semi-algebraic sets used to certify discreteness do not approximate the boundary well, and the complexity of the
trace polynomials involved makes computer implementation hard), they do provide some hope that there exist useful partial algorithms for the Geometric Recognition Problem. Such
algorithms may provide a useful stepping stone towards a fuller understanding of the general discreteness problem.

\subsection{Acknowledgments}
I thank Ari Markowitz for a detailed discussion of a false theorem and for general conversations about geometric group theory. I thank Gaven Martin, Jessica Purcell,
and Jeroen Schillewaert for helpful comments and discussion of various aspects of this work and surrounding mathematics.

\section{Representations of hyperbolic $3$-manifold groups}\label{sec:locglo}
We will freely identify $ \PSL(2,\C) $ with the group of orientation-preserving M\"obius transformations of the Riemann sphere $ \hat{\C} = \C \union \{\infty\} $, and with the group of
orientation-preserving isometries of hyperbolic $3$-space $ \H^3 $; the identifications are standard and correspond to the group action of $ \PSL(2,\C) $ on $ \hat{\C} $
given by
\begin{displaymath}
  \begin{bmatrix} a&b \\ c&d \end{bmatrix} \cdot z = \frac{az + b}{cz+d}
\end{displaymath}
and the induced action on upper-halfspace obtained from this conformal action by Poincar\'e extension. Let $ G $ be a geometrically finite, torsion-free, non-elementary,
discrete subgroup of $ \PSL(2,\C) $. These conditions imply that $ \H^3/G $ is a complete hyperbolic $3$-manifold with finite volume convex core. Conversely given any such manifold $ M $
development onto $ \H^3 $ induces a canonical representation $ \pi_1(M) \to \PSL(2,\C) $ (up to conjugacy) called the holonomy representation, and the image of this
is a discrete, geometrically finite, torsion-free, and non-elementary group. Working with parameterised generating sets as in \cref{ex:rileychar} and \cref{ex:12cbchar} picks out
a specific choice of representative for each conjugacy class, i.e.\ a section of a subset of the character variety in $ \Hom(G, \PSL(2,\C)) $. These sections can be equivalently
thought of as a family of subgroups of $ \PSL(2,\C) $ with a consistently marked generating set.

\begin{defn}
  A discrete subgroup of $ \PSL(2,\C) $ is called a \df{Kleinian group}~\cite{maskit,matsuzakitaniguchi}. A geometrically finite, torsion-free, non-elementary Kleinian group
  is called a \df{convex cofinite manifold group}. Given a Kleinian group $ G $, the maximal open subset of $ \hat{\C} $ on which $ G $ acts discontinuously will be
  denoted $ \Omega(G) $; so $ \Omega(G)/G $ is a (possibly empty, possibly disconnected, possibly singular) Riemann surface that is canonically identified with $ \partial \H^3/G $,
  the conformal boundary of $ \H^3/G $. The complement of $ \Omega(G) $ in $ \hat{\C} $ is the set of all accumulation points of $G$-orbits in $ \overline{\H^3} = \H^3 \union \hat{\C} $
  and is called the \df{limit set} of $ G $, denoted $ \Lambda(G) $.
\end{defn}

In this paper, we will essentially only be interested in Kleinian groups $G$ with $ \Omega(G) \neq \emptyset $; these are sometimes called \df{Kleinian groups of the second kind}.

Given an hyperbolic $3$-manifold $ M $ with $ \partial M \neq \emptyset $, it is possible to slightly deform the hyperbolic metric and obtain a new $3$-manifold $ \tilde{M} $. This corresponds to taking a
small deformation of the holonomy group $ G = \pi_1(M) $ and producing some isomorphic group $ \tilde{G} < \PSL(2,\C) $. The degrees of freedom of possible deformations are classified by
the following result of Marden~\cite{marden74g} (in the torsion-free setting) and Tukia~\cite{tukia85}; this version also incorporates
the earlier work of Ahlfors, Bers, Maskit, and others on the Teichm\"uller spaces of Kleinian groups.
\begin{thm}[{\cite[Theorems~3.25, 5.27, 5.28, and 5.32]{matsuzakitaniguchi}}]\label{thm:mardentukia}
  Let $ M = \H^3/G $ be a geometrically finite hyperbolic $3$-manifold and let $ \abs{M} $ be the underlying topological manifold. Decompose $ \partial M $ as a union of
  connected components $ \bigcup_{i=1}^n S_i $. The set of hyperbolic metrics on $ \abs{M} $ is parameterised by the product of the Teichm\"uller spaces $ \prod_{i=1}^n T(S_i) $;
  every choice of metric $ \tilde{M} $ is realised by some holonomy group $ \tilde{G} $, and the set of all possible holonomy groups forms a connected open subset of the character
  variety $ X(G) = \Hom(G,\PSL(2,\C))\git\PSL(2,\C) $. \qed
\end{thm}

The connected open subset which parameterises the hyperbolic metrics of $ \abs{M} $ is called the \df{quasiconformal deformation space} of $ G $, and is denoted $ \QH(G) $~\cite[\S 5.3]{matsuzakitaniguchi}. It is
equivalently defined as the set of conjugacy classes of representations $ \rho : G \to \PSL(2,\C) $ which satisfy three axioms:
\begin{enumerate}
  \item $\rho$ is faithful and $ \rho(G) $ is discrete;
  \item $\rho(g) $ is parabolic whenever $ g \in G $ is parabolic; and
  \item there exists a quasiconformal map $ \phi : \hat{\C} \to \hat{\C} $ such that $ \phi G\phi^{-1} = \rho(G) $.
\end{enumerate}
Note that (3) implies a strengthened version of (2), that $ \rho(g) $ is parabolic if and only if $ g $ is parabolic.

Algebraically, for each connected component $ S_i $ of $ \partial M $ we have a canonical map $ \pi_1(S_i) \to G $ given by inclusion of sets. This map is not faithful in general; there
might be compression discs in $ M $ with boundary on $ S_i $ which kill elements, and it could be that the surface $ S_i $ is  knotted nontrivially in $ M $ introducing further relations.
We write $ \Sigma_i $ for the image in $ G $ of this map. In general there is no simple algebraic decomposition of $ G $ in terms of the subgroups $ \Sigma_i $, but in special cases
one can be explicit enough about their relationship to find such a thing.

\begin{ex}
  In the special case that $ M $ is a compression body, there is a distinguished conformal end $ S_1 $ and $ \pi_1(M) = G $ is a quotient of $ \Sigma_1 $. There is a decomposition
  of $ G $ into a sequence of amalgamated products of the groups $ \pi_1(S_i) $ for $ i > 1 $, arising from cutting $ M $ along compression discs. This procedure arose first algebraically
  in Maskit's classification of function groups~\cite[Chapter~X]{maskit}.
\end{ex}

\begin{defn}\label{defn:rat_lam}
  A \df{rational lamination} on $ \Omega(G)/G $ is a choice of rational lamination (i.e.\ lamination with compact support) $ \Lambda_i $ on each surface $ S_i $ such that:
  \begin{enumerate}
    \item $ \Lambda_i $ is complete with respect to rank $1$ cusps, i.e.\ if $ S_i $ has a rank $1$ cusp then consider the deleted point as a loop on the surface with length $0$ and include
          it in $ \Lambda_i $.
    \item No leaves in $ \Lambda = \bigcup_{i=1}^n \Lambda_i $ are isotopic to other leaves, bound a compression disc, or bound a punctured disc in $ \H^3/G $. This is equivalent to asking
          that the leaves correspond via the map $ \pi_1(S_i) \to G $ to distinct nontrivial elements of $ G $ such that the image is parabolic if and only if the loop corresponds to a rank
          $1$ cusp on $ S_i $ itself as in (1).
  \end{enumerate}
\end{defn}

\begin{rem}
  An important special case is when every $ \Lambda_i $ is maximal on the surface $ S_i $, i.e.\ the complement $ S_i \setminus \Lambda_i $ is a union of thrice-holed spheres. The algebraic
  consequences of this setup for $\PSL(2,\C)$-character varieties of surface groups have been studied by Kabaya~\cite{kabaya11}.
\end{rem}

We will now explain how the fundamental groups of the conformal surfaces at the boundary of $ \H^3/G $ can be constructed from the much simpler fundamental
groups of the components of $ \bigcup (S_i \setminus \Lambda_i) $. This construction will be key in our later construction of conformal fundamental domains, since
our method is to patch together fundamental domains of the simpler groups to form a domain for the action of $ G $ on $\hat{\C} $.

\begin{cons}\label{cons:stratification_of_surfaces}
  For each $ i $, let $ \{P_{i,j}\}_{j=1}^{m_i} $ be the components of $ S_i \setminus \Lambda_i $. There exists a decomposition of $ \pi_1(S_i) $ (recall, we distinguish between $ \pi_1(S_i) $ and
  the image $ \Sigma_i $ of the surface group in $ G $) using combination theorems.

  Let $ \Gamma(\Lambda_i) $ denote the dual graph to the lamination $ \Lambda_i $, so the vertex set of $ \Gamma(\Lambda_i) $ is $  \{P_{i,j}\}_{j=1}^{m_i}  $ and the edge set is $ \Lambda_i $
  with incidence defined in the obvious way. Fix a maximal tree $ T_i $ for $\Gamma(\Lambda_i)$, topologically embedded on $ S_i $. For each $ P_{i,j} $, there are many possible conjugacy class
  representatives in $ \Sigma_i $ of $ \im (\pi_1(P_{i,j}) \to \pi_1(S_i) \to \Sigma_i) $, and we will simultaneously pick a representative for each in a way which is compatible with $ T $.
  We phrase the compatibility condition as a condition on the graph neighbourhood of a single $ P_{i,j} $, but we require it to be satisfied for all $ i $ and $ j $ so it is in reality a global condition. For
  every non-loop edge $ e $ adjacent to $ P = P_{i,j} $, if $ P^e $ is the vertex adjacent to $ P $ along $ e $ and if $\lambda$ is the leaf of $ \Lambda_i $ dual to $ e $:
  \begin{enumerate}
    \item If $ e $ lies in the distinguished tree $ T $, then we require the primitive elements of $ \im (\pi_1(P) \to \pi_1(S_i)) $ and $ \im (\pi_1(P^e) \to \pi_1(S_i)) $ which represent
          the loop $ \lambda $ to be equal.
    \item If $ e $ does not lie in the distinguished tree, then let $ \gamma_e $ be a loop in $ \pi_1(S_i) $ which represents the simple closed curve $ \tau \union e $ where $ \tau $ is the
          path in $ T $ from $ P $ to $ P^e $. We require the elements of  $ \im (\pi_1(P) \to \pi_1(S_i)) $ and $ \im (\pi_1(P^e) \to \pi_1(S_i)) $ representing the loop $\lambda$
          to be conjugate by $ \gamma_e $.
  \end{enumerate}
  For a loop edge, i.e.\ an edge $e$ joining some vertex $ P_i $ to itself, we define $ \gamma_e $ to be the word corresponding to a path on the surface $ P_i $ joining the two boundary
  components of $P_i$ obtained from the cut along the leaf on $S_i$ that is dual to $e$.

  A proof that there exist simultaneous compatible choices $ \Pi_{i,j} $ for conjugacy representatives of $ \im (\pi_1(P_{i,j}) \to \pi_1(S_i)) $ is a minor
  generalisation of results in Kabaya~\cite[\S3]{kabaya11}. Having made these choices, each surface group $ \Sigma_i $ is obtained from the  $ \pi_1(P_{i,j}) $ by amalgamated
  products (along the elements represented by edges in $ T $), HNN extentions (where the stable elements are the transformations $ \gamma_e $), and then a projection
  from $ \pi_1(S_i) $ to $ \Sigma_i $. It is not quite true that $\Sigma_i$ is obtained from the $ \Pi_{i,j} $ directly by amalgamated products and HNN extensions since
  the projection from $ \pi_1(S_i) $ to $ \Sigma_i $ (i.e.\ embedding the surfaces in $ G $) can trivialise the stable elements of the HNN extensions.
\end{cons}

The construction gives us a decomposition of $ G $ into local pieces represented by the diagram
\begin{displaymath}
  \begin{tikzcd}[row sep=1em]
    \Pi_{i,j} \arrow[r, hook] & \Sigma_i \arrow[r, hook] & G=\Pi_1(M)\\
    & \pi_1(S_i) \arrow[ur] \arrow[u,two heads]\\
    \pi_1(P_{i,j}) \arrow[ur] \arrow[uu, two heads]
  \end{tikzcd}
\end{displaymath}
where $ i $ ranges from $ 1 $ to $ n $ and where for each $ i $, $ j $ ranges from $ 1 $ to $ m_i $. From left to right
the algebraic maps along the diagonal require more and more global information to write down explicitly. Our results and constructions are primarily
concerned with the action of the group $ G $ on the Riemann sphere and not with the `global' action on $ \H^3 $, and so most of our interest
concentrates on the groups $ \Pi_{i,j} $ and their constituents since these are the things detected by fundamental domains of $ G $ on $ \hat{\C}$.
A great deal of literature does exist on the global algebraic decompositions of $ G $, see the monograph of
Aschenbrenner, Friedl, and Winton~\cite{aschenbrenner}.

\begin{ex}\label{ex:schottky_85}
  We give an explicit example where $ G $ is a maximal cusp group on the boundary of genus $2$ Schottky space. We first construct the group inside
  the character variety of $ F_2 $. The space of free groups in $ \PSL(2,\C) $ on two generators $ X $ and $ Y $ is a $3$-dimensional complex variety
  and may be parameterised by $ \tr X $, $ \tr Y $, and $ \tr XY $. We take the parameterisation
  \begin{displaymath}
    X = \begin{bmatrix} \frac{1}{2}(t_X + i t_{XY}) & -\frac{1}{2} (t_X + v) \\ -\frac{1}{2} (t_X - v) & \frac{1}{2}(t_X - i t_{XY}) \end{bmatrix}\;\text{and}\;
    Y = \begin{bmatrix} \frac{t_Y}{2} - i & \frac{t_Y}{2} \\ \frac{t_Y}{2} & \frac{t_Y}{2} + i \end{bmatrix}
  \end{displaymath}
  where $ v $ satisfies $ v^2 = 4 - t_{XY}^2 $; this is chosen so that $ \tr X = t_X $, $ \tr Y = t_Y $, and $ \tr XY = t_{XY} $. Algebraically, the parameterisation is
  a map $ \Phi : \C^4 \to \PSL(2,\C)^3 $, defined over $ \Q $ and with the coordinates of $ \C^4 $ labelled $ t_X, t_Y, t_{XY}, v $, restricted to the
  subvariety $ \V(v^2 + t_{XY}^2 - 4) \subset \C^4 $.

  \begin{figure}
    \labellist
    \small\hair 2pt
    \pinlabel {$U_2$} [l] at 85 27
    \pinlabel {$U_3$} [l] at 173 59
    \pinlabel {$U_1$} [r] at 168 85
    \endlabellist
    \centering
    \includegraphics[width=0.6\textwidth]{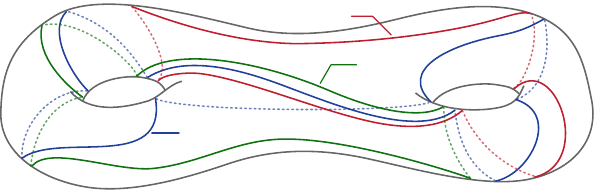}
    \caption{The curves on the genus $2$ surface represented by the three words of \cref{ex:schottky_85}. The full surface is $ S_1 $, and when it is cut along the maximal lamination
    it falls into two thrice-holed spheres, $ P_{1,1} $ and $ P_{1,2} $. The graph $ \Gamma(\Lambda) $ is a $\theta$-curve, and the maximal tree $ T_1 $ can be chosen to consist
    of a vertex interior to each thrice-holed sphere joined by a single edge intersecting the $U_2$ curve once.\label{fig:eightfive_schottky}}
  \end{figure}

  If $ F_2 = \langle X, Y \rangle $ is identified with the genus $2$ handlebody group such that $ X $ and $ Y $ are the cores of the two handles, then the three words
  \begin{displaymath}
    U_1 = X^{-1}Y^2,\,
    U_2 = X^{-1}Y^3X^{-2},\,\text{and}\,
    U_3 = YX^{-2}
  \end{displaymath}
  represent simple closed curves on the boundary of the manifold. These curves form the three leaves of a maximal rational lamination $ \Lambda $ on the single surface end
  of the handlebody, shown in \cref{fig:eightfive_schottky}. The motivation behind this particular choice of words is explained in~\cite[Example~3.7]{elzenaar25h} but is not material to our current discussion.

  It is not hard to find a representation of $ F_2 $ such that the three words $ \{U_i\}_{i=1}^3 $ are parabolic: the isolated solution to the system of
  equations $ \tr^2 U_1 = \tr^2 U_2 = \tr^2 U_3 = 4 $ approximated by
  \begin{displaymath}
    (t_X, t_Y, t_{XY}) = (0.7607 + 0.8579i, -0.7610 - 0.8579i, 2.3146 - 2.6103i)
  \end{displaymath}
  gives a representation with discrete image $ G $ so that $ \H^3/G $ is a handlebody with boundary two thrice-punctured spheres, $ P_{1,1} $ and $ P_{1,2} $. In our notation the group has a single
  conformal surface $ S_1 $ (of genus $2$, corresponding to plumbing\footnote{Given a rank $1$ cusp in a geometrically finite group joining two surfaces,
  we may `plumb' the surfaces across the cusp by cutting out a punctured disc neighbourhood of the cusp from both and stitching together the resulting boundary curves~\cite{kra90,maloni10}.}
  the two surfaces $ P_{1,i} $ together), and the end group $ \Sigma_1 $ which is the image of $ \pi_1(S_1) $ in $ G $ is
  the entire group $ G $. Let $ (-)^\dagger $ denote reversal of a word, and consider the two subgroups $ \Pi_{1,1} = \langle U_1, U_2 \rangle $ and $ \Pi_{1,2} = \langle U_1^\dagger, U_2^\dagger \rangle $.
  It is not too hard to see, e.g.\ by curve coding techniques, that these represent images of $ \pi_1(S_{1,1}) $ and $ \pi_1(S_{1,2}) $ inside $ G $. If we conjugate $ \Pi_{1,2} $ by $ X^{-1}Y $
  and call the result $ \Pi_{1,2}' $, then $ \Pi_{1,2}' \inter \Pi_{1,1} = \langle U_2 \rangle $ and the full group $ G $ is obtained by two HNN extensions of the amalgamated
  product of $ \Pi_{1,2}' *_{\langle U_2 \rangle} \Pi_{1,1} $ followed by killing the compression disc boundary curves in the genus $2$ handlebody.
\end{ex}

\section{Peripheral groups and the convex core boundary}\label{sec:peripheral}
In the previous section we studied algebraic decompositions of a convex cofinite manifold group $G $ arising from topological decorations on the conformal boundary of $ \H^3/G $.
These decompositions are useful in studying the algebraic structure of the character variety, but since they are purely topological they do not give much information on the structure
of subsets of the character variety of interest to geometric topologists, most notably the discreteness locus. In this section we introduce machinery to isolate loci within the
character variety within which there are algebraic decompositions compatible with the geometry.

\begin{defn}
  A subgroup $ \Pi \leq G $ is called \df{peripheral} if there exists a non-empty open simply connected subset $ U \subset \Omega(G) $ left invariant by $ \Pi $.
  Since $ \Lambda(\Pi) \subset \Lambda(G) $, $ U $ is necessarily a subset of $ \Omega(\Pi) $. We say that $ \Pi $ is \df{F-peripheral}
  if it is Fuchsian and $ U $ is one of the discs on which it acts. We write $ \Delta(\Pi) $ for the maximal topological disc $U$ preserved by $\Pi$;
  it is called the \df{peripheral disc} of $\Pi$. The set of maximal\footnote{If `maximal' is deleted we still obtain a graph, but it is much larger as
  every subgroup of an $F$-peripheral group is itself $F$-peripheral.} $ F$-peripheral subgroups of $ G $ forms the vertices of a natural graph $ \FPer(G) $, where the incidence
  relation is inclusion (two subgroups lie in a face if they intersect nontrivially). The overgroup $ G $ acts on this graph by conjugation.
\end{defn}

\begin{history}
  This definition was introduced to the literature by Keen and Series~\cite{keen94} as a technical tool to deal with the combinatorial circle-packing
  structure of limit sets (an example of such a circle-packing is shown in \cref{fig:schottky_85} below) which they had previously used to study a deformation
  space of punctured torus groups~\cite{keen93}. They attributed the insight that the circle patterns in limit sets (which they called circle chains,
  see \cref{defn:chain} below) can be controlled via less ad-hoc algebraic methods to David Wright. The notion of a peripheral subgroup of a hyperbolic group
  $G$ as a subgroup whose Cayley graph is an `extremal' subset of the ends of the Cayley graph of $ G $ originated with Gromov and
  is of interest in geometric group theory~\cite{bowditch01,haissinsky16,benzvi22}.
\end{history}

\begin{ex}\label{ex:apollonian}
  If $G$ is the group generated by the two parabolics $ X(z) = z+1 $ and $ Y(z) = z/(2iz+1) $, then $ \H^3/G $ is topologically a genus $2$ handlebody.
  The commutator $ [X,Y] $ is parabolic, and the three words $ X $, $ Y $, and $ [X,Y] $ represent three rank $1$ cusps and three disjoint simple closed curves
  on the topological genus $2$ surface obtained by plumbing the boundary of $ \H^3/G $.

  There are two conjugacy classes of F-peripheral groups in $ G $. They are represented by the two $ (\infty,\infty,\infty)$-triangle
  groups $ \langle X^{-1}, Y^{-1}XY \rangle $ and $ \langle X^{-1}Y^{-1}X, Y \rangle $. One possible fundamental domain for the action of $ G $ on $ \FPer(G) $ is
  \begin{displaymath}
    \begin{tikzcd}[column sep=huge, row sep=small]
      \langle X, YX^{-1}Y^{-1} \rangle \arrow[d,dash,"\langle X \rangle"]\\
      \langle X^{-1}, Y^{-1}XY \rangle \arrow[r, dash, "\langle X^{-1}Y^{-1}XY \rangle"] & \langle X^{-1}Y^{-1}X, Y\rangle \arrow[d, dash, "\langle Y \rangle"]\\
      &\langle Y^{-1}, XYX^{-1} \rangle.
    \end{tikzcd}
  \end{displaymath}
  Since $Y$ acts to identify the two vertices on the left and $ X $ acts to identify the two vertices
  on the right, the quotient $ \FPer(G)/G $ is the graph
  \begin{displaymath}
    \begin{tikzcd}[column sep=huge]
      \langle X^{-1}, Y^{-1}XY \rangle \arrow[loop left,dash,"\langle X \rangle"]\arrow[r, dash, "\langle X^{-1}Y^{-1}XY \rangle"] & \langle X^{-1}Y^{-1}X, Y\rangle \arrow[loop right, dash, "\langle Y \rangle"]
    \end{tikzcd}
  \end{displaymath}
  which can be identified with the dual graph to the maximal rational lamination defined by $ X $, $ Y $, and $ [X,Y] $.
\end{ex}
\begin{rem}
  In this remark we use terminology from Dicks and Dunwoody~\cite[\S\S I.1--I.2]{dicks} (but with obvious modifications to account for our graphs being undirected).
  We will not use any technical machinery from the theory of groups acting on graphs in this paper, but readers already familiar with this
  field may find it useful to note that our `fundamental domains' (as in \cref{ex:apollonian}) are obtained by taking a $G$-traversal for the action
  of $ G $ on $ \FPer(G) $ and then adjoining to every `hanging' edge its other vertex; i.e.\ the graph closure of the $G$-traversal. We will usually
  (in light of \cref{cons:stratification_of_surfaces}) take all $G$-traversals to be \emph{fundamental}, i.e.\ the graph interior of the traversal
  is a subforest of $ \FPer(G) $ such that every component of the traversal lies in a distinct connected component of $ \FPer(G) $; this is
  possible by~\cite[Proposition I.2.6]{dicks}.
\end{rem}

The peripheral discs of the conjugates of an F-peripheral group form a series of round open discs in $ \Omega(G) $, and the domes above
these open discs support some of the flat pieces of the pleated hyperbolic surface $ \hconv \Lambda(G) $ (the hyperbolic convex hull of $ \Lambda(G) $~\cite[\S 3.1.1]{matsuzakitaniguchi}).

\begin{rem}
  The relationship of the convex core bending lamination to the structure of the quasiconformal deformation space is well established. Choi and Series have proved, using
  the Hodgson--Kerchkoff theory of cone manifold deformations, that normalised lengths of leaves of the bending lamination give a coordinate system on the space of convex
  structures of hyperbolic $3$-manifolds~\cite{choi06}. A related conjecture usually attributed to Thurston is that the holonomy representation of a hyperbolic manifold
  is determined exactly by the angles across of the bending lamination of its convex core boundary. In certain special cases, bounds on the relationship between pleating
  angles and pleating lengths are known, for instance see Miyachi~\cite[Lemma~7.1]{miyachi03}. Bonahon and Otal have proved~\cite{bonahonotal04} the existence of quasi-Fuchsian
  groups that realise every possible bending lamination on their two ends, and proved uniqueness of these groups in some cases; an alternative existence proof was given
  by Baba and Ohshika using model manifold theory~\cite{baba23}. Proofs of Thurston's conjecture have been announced for convex cocompact groups by Dular and Schlenker in 2024~\cite{dular24}
  and for singly degenerate manifolds on the boundary of quasi-Fuchsian space by Dular~\cite{dular25}. It remains open for arbitrary Kleinian groups. Of additional interest to
  us is work of Series~\cite{series06} who proved, using the theory of peripheral structures and their limits, both existence and uniqueness of groups realising every measured
  lamination for the special case of representations of $ \pi_1(S_{0,4}) $ and $ \pi_1(S_{1,1}) $.
\end{rem}

\begin{defn}\label{defn:chain}
  Let $ \Lambda = \bigcup_{i=1}^n \Lambda_i $ be a rational lamination of $ \Omega(G)/G $ as described in \cref{defn:rat_lam}; let $ T $ be a maximal forest of $ \Gamma(\Lambda) $
  and use it to define a family of subgroups $ \Pi_{i,j} $ of $ G $ via the procedure described in \cref{cons:stratification_of_surfaces}.
  If the subgroups $ \Pi_{i,j} $ are all maximal F-peripheral subgroups, then we say that they form a \df{$\Lambda$-circle chain} in $G$. Further,
  if $ \tilde{G} $ is the image of a representation $ \rho : G \to \PSL(2,\C) $ then we say that $ \tilde{G} $ has a $ \Lambda$-circle chain
  if the images $ \rho(\Pi_{i,j}) $ are maximal F-peripheral and have the same incidence structure as the groups $ \Pi_{i,j} $.
\end{defn}

\begin{ex}\label{ex:riley}
  Keen and Series studied the case that the group $ G $ is free on two parabolic generators and $ \Omega(G)/G $ is a four-punctured sphere (i.e.\ lies in the Riley slice, the locus of
  free quasiconformally deformable groups in the parameter space of \cref{ex:rileychar})~\cite{keen94}, and the case that the group $ G $ is freely generated by a parabolic element
  and a loxodromic element which have a parabolic commutator (i.e.\ lies in the Maskit slice)~\cite{keen93}. \Cref{ex:apollonian} above gives an example of
  a circle chain in a group on the boundary of the Riley slice.
\end{ex}

\begin{figure}
  \labellist
  \hair 2pt
  \pinlabel {$\Pi_{1,2}$} at 220 440
  \pinlabel {$\Pi'_{1,2}$} at 152 241
  \pinlabel {$\Pi_{1,1}$} at 291 166
  \endlabellist
  \centering
  \includegraphics[width=.6\textwidth]{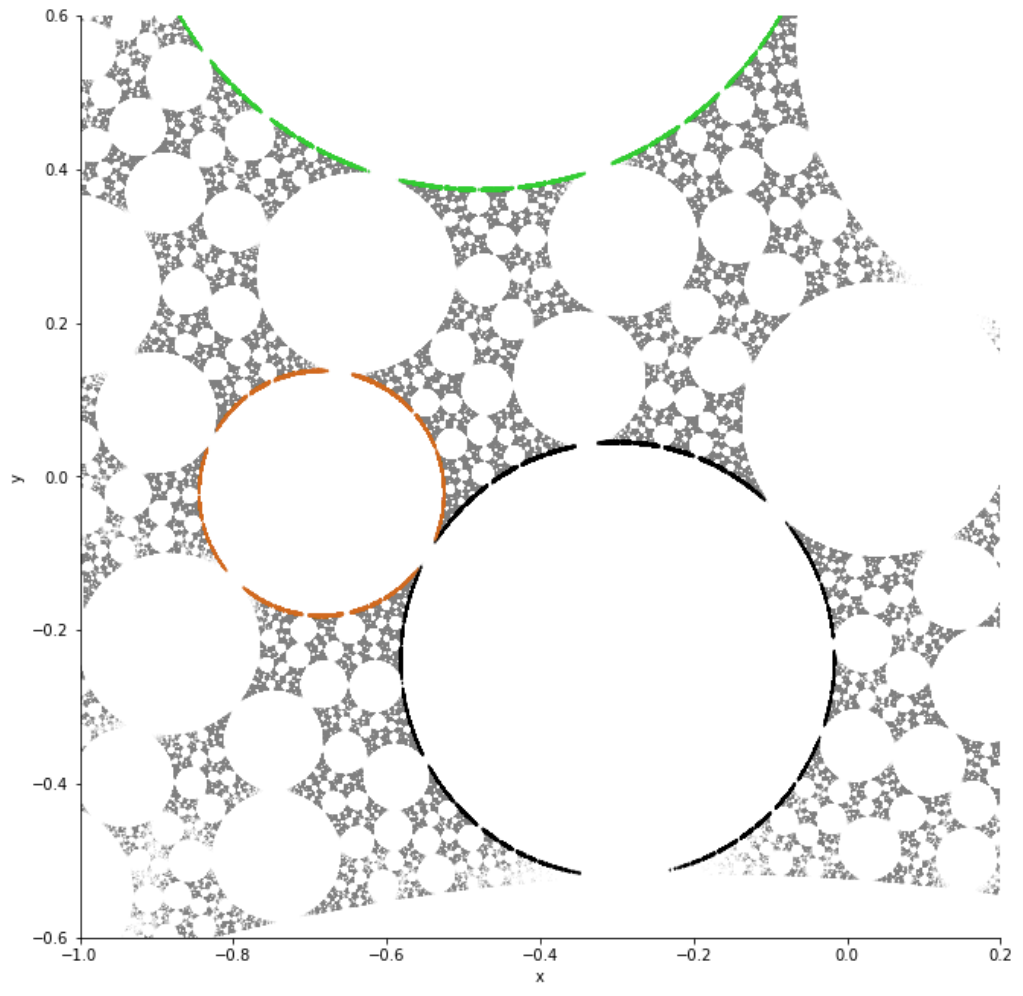}
  \caption{The peripheral discs for the subgroups $ \Pi_{1,1} $, $ \Pi_{1,2} $, and $ \Pi_{1,2}' $ of $ G $ defined in \cref{ex:schottky_85}. The graph $ \FPer(G) $ can be identified with the tangency graph of the circle packing $ \hat{\C} \setminus \Lambda(G) $.\label{fig:schottky_85}}
\end{figure}

\begin{ex}
  The groups $ \Pi_{1,1} = \langle U_1, U_2 \rangle $ and $ \Pi_{1,2} = \langle U_1^\dagger, U_2^\dagger \rangle $ of \cref{ex:schottky_85} are $F$-peripheral and lie in different
  conjugacy classes in the group $ G $ of that example. They do not intersect, so do not form a circle chain (since by definition a circle chain comes from a connected subgraph of $ \FPer(G) $
  for each surface); the two groups $ \Pi_{1,1} $ and $ \Pi_{1,2}' $ do form a circle chain. The limit sets of all three F-peripheral groups are shown superimposed on the limit set of $ G $
  in \cref{fig:schottky_85}.
\end{ex}

\begin{figure}
  \centering
  \includegraphics[width=.6\textwidth]{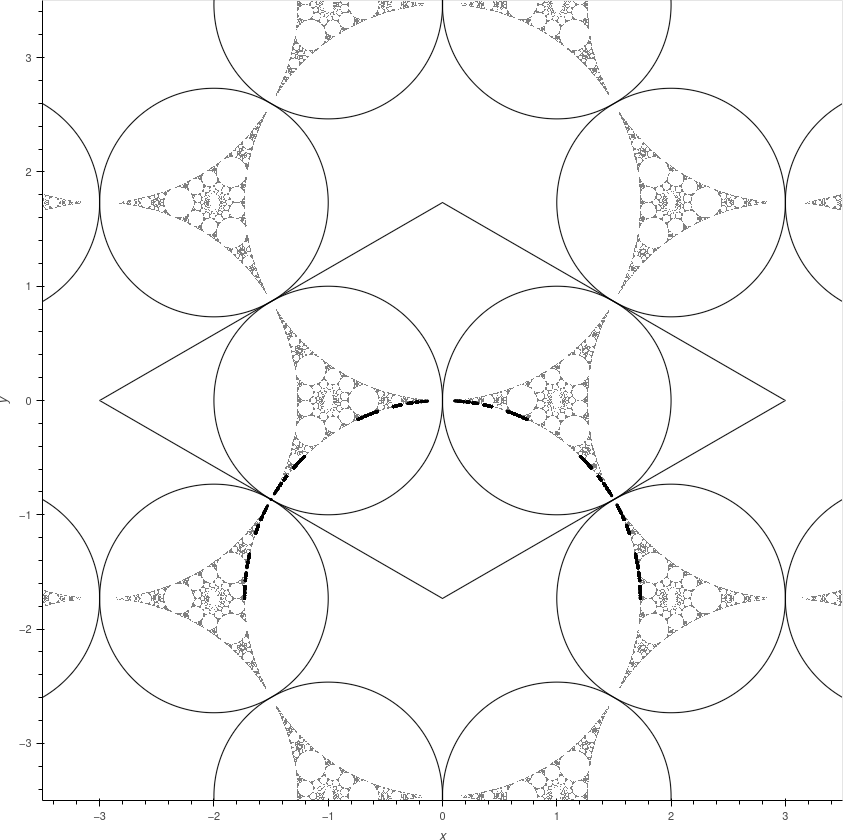}
  \caption{In grey, we show the limit set of a $ (1;2)$-compression body cusp group $G$. The subgroup $\Pi$ with highlighted limit set is F-peripheral, but not maximal.\label{fig:circles_pi3}}
\end{figure}

\begin{ex}\label{ex:non_maximal}
  We give an example where the lamination $ \Lambda $ is not maximal. Consider the group $G$ defined by choosing $ \alpha = 3+i\sqrt{3} $, $ \beta = 3-i\sqrt{3} $, and $ \lambda = 1 $
  in the parameterisation of \cref{ex:12cbchar}. This group is the $ \theta = \pi/3 $ circle pattern group studied in our earlier work~\cite[Example~2.1]{elzenaar25h}, and lies on the
  boundary of $ (1;2)$-compression body space. The subgroup $ \Pi = \langle MP^{-1} MQ^{-1}, M \rangle $ is F-peripheral, but is not maximal. This can be seen in \cref{fig:circles_pi3},
  where the limit set $\Lambda(\Pi) $ covers only part of the boundary of a peripheral disc of $G$; in the language of Keen and Series~\cite{keen94}, it is not \textit{strongly} F-peripheral.
  The maximal F-peripheral group which contains $ \Pi $ is $ \Pi' = \langle \Pi, PQ^{-1} M QP^{-1} \rangle $; the limit set $ \Lambda(\Pi') $ consists of every limit point of $G$ that
  lies on the boundary of the disc preserved by $ \Pi $. This maximal F-peripheral group is a Fuchsian four-punctured sphere group: the corresponding lamination has only
  two leaves on the topological genus $2$ surface.
\end{ex}

The goal of the remainder of this section is to show that the quotient of $ \FPer(G) $ by the conjugation action is equal to the graph $ \Gamma(\Lambda) $, when there is
a $ \Lambda$-circle chain; in fact a $ \Lambda$-circle chain is identified with a connected fundamental set for the action of $ G $ on vertices. An alternative way of putting
this is that if $ \Lambda $ is a lamination which is compatible with the geometry of the group, then there is a well-defined covering map $ \FPer(G) \to \Gamma(\Lambda) $.
This is a consequence of the following proposition, which states that knowing a $ \Lambda$-circle chain in $ G $ is equivalent to knowing the entire combinatorial convex core
angle structure.

\begin{prp}\label{prp:circle_chain_encodes_membership}
  Let $ G $ be a convex cofinite manifold group with $ \Omega(G) \neq \emptyset $ and let $ \Lambda $ be the convex core bending lamination of $ \Omega(G)/G $. Suppose that $ \rho : G \to \PSL(2,\C) $
  is discrete, non-elementary, and admits a $\Lambda$-circle chain. If $ \H^3/\rho(G) $ is homeomorphic to $\H^3/G $, then $ \rho(G) \in \QH(G) $.
\end{prp}

\begin{rem}
  In particular, if $ X(G) $ is known to have a single locus of discrete groups with a given conformal boundary structure---e.g.\ the case of the Riley and Maskit slices of \cref{ex:riley},
  or more generally genus $2$ Schottky space, since all genus $2$ handlebodies are homeomorphic---then existence of a $\Lambda$-circle chain gives a full certification of membership of $ \QH(G) $.
\end{rem}

The proof of \cref{prp:circle_chain_encodes_membership}, which will be completed in \cref{lem:pl_ray_is_glued_discs}, goes via the following chain of reasoning: if there exists a $\Lambda$-circle chain then the image
under $ \rho $ of each peripheral group $ \Pi_{i,j} \leq G $ acts as a hyperbolic isometry group on the hyperbolic dome $ \Delta_{i,j} $ above its peripheral disc, and the
union of the quotients $ \Delta_{i,j}/\Pi_{i,j} $ is the end of the convex core of $ \H^3/\rho(G) $ facing the surface $ S_i $. In other words, $\Lambda$-circle
chains model the convex core boundary of the manifold. Our arguments generalise those given by Keen and Series~\cite[Lemma~3.5]{keen94}, with two
additional difficulties: first, we allow more complicated global groups $G$ than quotients of the genus $2$ surface group; and second, we allow arbitrary laminations,
not just maximal ones.

In the following sequence of lemmata we assume the hypotheses of \cref{prp:circle_chain_encodes_membership}, so $ \rho(G) $ is discrete and admits a $\Lambda$-circle chain $ \{\Pi_{i,1},\ldots,\Pi_{i,m_i}\}_{i=1}^n $.
\begin{lem}\label{lem:strongly_peripheral}
  For all $ i,j $, $ \Lambda(\Pi_{i,j}) = \Lambda(\rho(G)) \inter \overline{\Delta(\Pi_{i,j})} $ where $ \Delta(\Pi_{i,j}) $ is the peripheral disc preserved by $ \Pi_{i,j} $.
\end{lem}
\begin{proof}
  The quotient $ \Delta(\Pi_{i,j})/\Pi_{i,j} $ is a hyperbolic surface with $ b $ boundary components represented by the hyperbolic elements $ g_1,\ldots,g_b \in \Pi_{i,j} $ and $ p $ punctures
  represented by parabolic elements $ g_{b+1},\ldots,g_{b+p} \in \Pi_{i,j} $. By the definition of a circle chain, for each $ g_k $ there exists some $ j_k \in \{1,\ldots,m_i\} $ and
  some $ \gamma_k \in \rho(G) $ such that $ g_k \in \Pi_{i,j} \inter \gamma \Pi_{i,j_k} \gamma^{-1} $. The boundary circle $ \partial \Delta(\Pi_{i,j}) $ intersects with the boundary
  circle $ \partial \Delta(\gamma \Pi_{i,j_k} \gamma^{-1}) $ at the fixed points of $ g_k $. If $ g_k $ is hyperbolic (i.e.\ $k \leq b $), then let $ \sigma_k $ be the arc
  in  $ \partial \Delta(\Pi_{i,j}) $ between the two fixed points of $ g_k $ which is contained in $ \Delta(\gamma \Pi_{i,j_k} \gamma^{-1}) $; it contains no limit points of $ \rho(G) $
  since $ \gamma \Pi_{i,j_k} \gamma^{-1} $ is F-peripheral.

  Now by standard theory of Fuchsian groups~\cite[\S 10.3]{beardon} every interval of discontinuity of $ \Pi_{i,j} $ on $ \partial \Delta(\Pi_{i,j}) $ is $\Pi_{i,j}$-equivalent
  to the interval between fixed points of $ \gamma \Pi_{i,j_k} \gamma^{-1} $ for some F-peripheral group in the conjugacy class of a circle chain element.
  Since points on $ \partial \Delta(\Pi_{i,j}) $ are either in intervals of discontinuity or are limit points of $ \Pi_{i,j} $, and no limit points of $ \rho(G) $ can lie on any images of $ \sigma_k $,
  we see that $ \partial \Delta(\Pi_{i,j}) \inter \Lambda(\rho(G)) = \Lambda(\Pi_{i,j}) $ as required.
\end{proof}

\begin{lem}\label{lem:neilsen_region_is_fd}
  For each $ i,j $ let $ H_{i,j} $ be the hyperbolic plane erected above $ \Delta(\Pi_{i,j}) $; by Poincar\'e extension from $ \hat{\C} $ to $ \H^3 $, $ \Pi_{i,j} $ acts
  as a group of hyperbolic isometries on $ H_{i,j} \simeq \H^2 $ and so we may define a Nielsen region $ N_{i,j} $ for this action. This Nielsen region is precisely
  invariant under $ \Pi_{i,j} $ in $ \rho(G) $.
\end{lem}
\begin{proof}
  We first show that $ \Pi_{i,j} = \Stab_{\rho(G)} N_{i,j} $. Let $ S = \Stab_{\rho(G)} N_{i,j} $. The subgroup $ S $ stabilises the boundary of $ \Delta(\Pi_{i,j}) $.
  Indeed, $S$ sends an arc through two points in $ N_{i,j} $ to another arc through two points in $ N_{i,j} $, and since $ N_{i,j} $ is full-dimensional in $ H_{i,j} $, for
  any pair of points $ \xi_1,\xi_2 $ in $ \partial \Delta(\Pi_{i,j}) $ there exist two points in $ N_{i,j} $ lying on the geodesic $[\xi_1,\xi_2]$; thus $ S $, sending $ \xi_1 $
  and $ \xi_2 $ to the endpoints of the geodesic joining the images of the two points in $ N_{i,j} $, sends $ \xi_1 $ and $ \xi_2 $ to two other points on $ \partial \Delta(F_{i,j}) $.
  Since $ S $ stabilises the boundary, and therefore stabilises the entire hemisphere $ H_{i,j} $, it is Fuchsian. Thus, since $ S \geq \Pi_{i,j} $, there is an induced covering
  map $ N_{i,j}/\Pi_{i,j} \to N_i/S $. The maximality of $ \Pi_{i,j} $ completes the proof that $ \Pi_{i,j} = S $.

  Now suppose for contradiction that $ g \in \rho(G)\setminus \Pi_{i,j} $ but $ g(N_{i,j}) \inter N_{i,j} \neq \emptyset $. We have two cases.
  \begin{enumerate}
    \item Suppose $ \Delta(\Pi_{i,j}) = g \Delta(\Pi_{i,j}) $. Then $ g $ stabilises $ \Delta(\Pi_{i,j}) $, and by \cref{lem:strongly_peripheral} it must permute the arcs of
          discontinuity of $ \Pi_{i,j} $ on the boundary. Further it is conformal on $ H_{i,j} $ and so preserves the angles of the edges of $ N_{i,j} $ on translation. These
          two facts imply that $ g $ stabilises the Nielsen region and thus $ g \in \Pi_{i,j} $, giving the required contradiction.
    \item On the other hand, suppose $ \Delta(\Pi_{i,j}) \neq g\Delta(\Pi_{i,j}) $. The set $ g(N_{i,j}) \inter N_{i,j} $ lies on the arc of intersection of the two geodesic
          domes $ H_{i,j} $ and $ gH_{i,j} $. Since $ \Pi_{i,j} $ is F-peripheral, the arc $ \Delta_{i,j} \inter \partial g\Delta_{i,j} $ is an arc of discontinuity for $ \Pi_{i,j} $
          and so $ N_{i,j} $ is bounded by the arc joining the intersection points of $ \Delta_{i,j} $ with $ \partial g\Delta_{i,j} $. On the other hand $ N_{i,j} $ is open,
          and so $ N_{i,j} $ cannot contain any points of this arc (which gives the contradiction). \qedhere \end{enumerate}
\end{proof}

\begin{lem}\label{lem:pl_ray_is_glued_discs}
  The convex core boundary $ \partial (\hconv \Lambda \rho(G))/\rho(G) $ consists of $r = \sum_{i=1}^n m_i $ flat pieces, glued along the pleating locus of the surface which consists
  exactly of the projection of the axes of the boundary-parallel generators of the $ \Pi_{i,j} $ (i.e.\ the words representing the leaves of $ \Lambda $).
\end{lem}
\begin{proof}
  Let $ H_{i,j} $ be the hemispheres above the $ \Delta(\Pi_{i,j}) $ and let $ N_{i,j} $ be the respective Nielsen regions for the actions on the hemispheres by the $\Pi_{i,j} $.
  Each $ S_{i,j} \coloneq N_{i,j}/F_{i,j} $ is a thrice-holed sphere (possibly with some holes represented by parabolics). All surfaces are hyperbolic and so have curvature $ -1 $.
  We can therefore compute the area of each $ S_{i,j} $ using the Siegel area formula~\cite[Theorem~10.4.3]{beardon}. By \cref{lem:neilsen_region_is_fd}, $ N_{i,j}/\Pi_{i,j} = N_{i,j}/\rho(G) $
  for all $ i $. Since all the $ \Pi_{i,j} $ are F-peripheral, each $ N_{i,j}/\rho(G) $ lies in $ \partial (\hconv \Lambda( \rho(G)))/\rho(G) $; and since the $ \Pi_{i,j} $ are
  non-conjugate, they are disjoint subsets of the surface. On the other hand, we may apply the Siegel area formula to the full surface, and the result is the same as the sum of the areas of the surfaces.
  Thus the full surface must be the union of the $S_{i,j} $.
\end{proof}

\begin{rem}
  This concludes the proof of \cref{prp:circle_chain_encodes_membership}.\qed
\end{rem}

% An immediate application is the construction of interesting group extensions of infinite covolume groups to produce cofinite groups with embedded Fuchsian surfaces. We state the
% result for maximal laminations on genus $2$ surfaces only, and let the reader supply the generalisations to higher genus surfaces and manifolds with more conformal ends; the proof
% is just an application of Maskit's second combination theorem~\cite[\S VII.E]{maskit}:
% \begin{thm}
%   Let $ G $ be a convex cofinite manifold group and let $ M = \H^3/G $. Suppose that $ G $ is maximally cusped (i.e.\ all components of $ \partial M $ are unions of thrice-punctured spheres joined at rank $1$ cusps) such that
%   that the connected components of $ \Gamma(\Lambda) $ all have two vertices (i.e.\ all components of $ \partial M $ are genus $2$ when the punctures are filled in). For each component
%   $ S_i $ of $ \partial M $, there exists some parabolic $ \Phi_i \in \PSL(2,\C) $ so that $ \Phi_i $ conjugates $ \Pi_{i,1} $ onto $ \Pi_{i,2} $ and sends the exterior of $ \Delta(\Pi_{i,1}) $
%   to the interior of $ \Delta(\Pi_{i,2}) $. The extension $ \hat{G} = \langle G, \Phi_1, \ldots, \Phi_n \rangle $ is discrete and $ \H^3/\hat{G} $ is a hyperbolic $3$-manifold of finite volume; this volume is equal to the
%   convex core volume of $ M $. \qed
% \end{thm}
%
% A large number of special cases of this theorem and its generalisations appear as various constructions of Wielenberg~\cite{wielenberg78} and Apanasov~\cite[325--326]{apanasov}.

\section{Pleating varieties}\label{sec:varieties}
Given a rational lamination $ \Lambda $ on the conformal boundary of a finite-type topological $3$-manifold $ M $, we have defined algebraic structures
called $\Lambda$-circle chains in $ G = \pi_1(M) $ which control the large-scale geometry of the surface ends of convex cofinite manifold
representations $ \rho : G \to \PSL(2,\C) $; equivalently, they control the bending lamination on the convex core boundary of $ \H^3/\rho(G) $.

\begin{defn}\label{defn:pleating_variety}
  Let $ G $ be a convex cofinite manifold group, let $ \Lambda $ be a maximal rational lamination on $ \Omega(G)/G $, and let $ X $ be an algebraic parameter space for the
  representations $ \rho : G \to \PSL(2,\C) $. The \df{$\Lambda$-pleating variety} is the set $ \mc{P}_\Lambda \subset X $
  of $ \rho \in X $ such that $ \rho(G) $ has a $\Lambda$-circle chain. Here, as we will often do, we identify points of $X$ with their corresponding images in $ X(G) $, forgetting the data of the parameterisation.
\end{defn}
\begin{rem}
  Our definition differs slightly from that of Choi and Series~\cite{choi06} as they do not require circle chains to consist of maximal F-peripheral groups.
\end{rem}
\begin{rem}
  We will usually work with parameter spaces $ X $ rather than the character variety $ X(G) $ since it is convenient when constructing examples. For example, the Riley slice (\cref{ex:riley})
  contains free groups on two generators, so embeds into a three-dimensional character variety; but it also embeds into a one-dimensional linear slice where two group elements are fixed parabolic. In this
  situation it is much easier to work with the one-parameter subvariety.
\end{rem}

\begin{figure}
  \labellist
  \hair 2pt
  \pinlabel {$\Lambda$} [b] at 115 170
  \pinlabel {$\Lambda'$} [b] at 277 156
  \pinlabel {cusp} [t] at 193 36
  \pinlabel {exterior of $ \QH(G) $} at 68 41
  \endlabellist
  \centering
  \includegraphics[width=.6\textwidth]{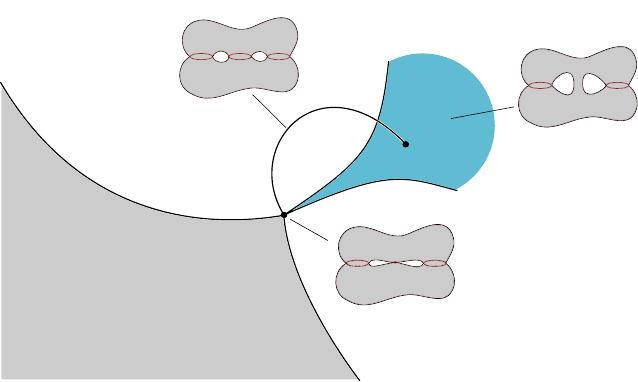}
  \caption{The relative position of two pleating varieties in $ \mc{S}_2 \subset X(F_2) $. \label{fig:pleating_ray_strata}}
\end{figure}

The object of this section is to describe the geometry of the pleating varieties in $ X $. The most na\"ive observation one makes is that a representation $ \rho \in X $ lies on the $ \Lambda$-pleating variety
only if the words in $ G $ corresponding to all the leaves of $ \Lambda $ are real (a more precise version of this is \cref{prp:schottky_rays} below). There are two immediate
obstructions to the converse holding: first, it is possible for the Fuchsian groups generated by these words (which are putative F-peripheral groups) to fail to be maximal
Fuchsian groups; and second, it is possible for there to exist maximal Fuchsian subgroups, generated by words representing $ \Lambda $, which are not peripheral. We will see that
the edge of the pleating variety is detected by the first type of obstruction. That is, suppose $ G $ is a convex
cofinite manifold group with a $\Lambda$-circle chain and that $ \sigma : [0,1] \to \QH(G) $ is a path such that $ \sigma(0) $ is the identity representation and such
that $ \sigma(1)(G) $ does not admit a $\Lambda$-circle chain; then the infimum $ t_\beta $ of $ t \in [0,1] $ such that $ \sigma(t)(G) $ does not have a $ \Lambda$-circle chain is of one of two kinds:
\begin{itemize}
  \item it does not admit any circle chain, i.e.\ the bending lamination of the convex core of $ \H^3/\sigma(t_\beta)(G) $ is not rational; or
  \item it admits a $ \Lambda'$-circle chain where $ \Lambda' $ is obtained by deleting some leaves from $ \Lambda $, and this circle chain is produced by taking
        the image of any $ \Lambda$-circle chain in $ G $ under the map $ \sigma(t_\beta) $ and, for every leaf $ \lambda $ that is deleted from $ \Lambda $, replacing
        the F-peripheral groups $\Pi $ and $ \Pi' $ joined by an edge dual to $\lambda $ with the group $ \langle \Pi, \Pi' \rangle $ iteratively (if a peripheral group
        is incident to two deleted leaves then the new group replacing it is generated by all three groups incident with these leaves, etc.); the new groups produced
        after the iterative procedure terminates (all leaves are removed) is now maximal and F-peripheral.
\end{itemize}

\begin{ex}
  \Cref{fig:pleating_ray_strata} is a schematic illustrating the relative embeddings of $ \mc{P}_{\Lambda} $ and $ \mc{P}_{\Lambda'} $ in a slice of genus two Schottky space
  where the two core curves of the handlebody are fixed hyperbolic. These curves define a lamination $ \Lambda' $ with two leaves that can be completed to a maximal lamination $ \Lambda $ by adding
  an additional curve $\lambda$. The $2$-dimensional pleating variety for $ \Lambda' $ is an embedding of a subset of the Teichm\"uller space of a four-holed sphere, viewed as the genus $2$-surface
  cut along the $\Lambda'$. This variety meets the boundary of the quasiconformal deformation space at the cusp corresponding to pinching $\lambda$ to a parabolic.
  Also eminating from this cusp is the $1$-dimensional pleating variety for $\Lambda$; the pleating angle across $ \lambda$
  increases from $0$ at the cusp to $ \pi$ where the two F-peripheral subgroups in the circle chain cease to be maximal and the $1$-dimensional variety hits the $2$-dimensional variety.
\end{ex}

Let $ \mc{S}_g $ denote genus $ g $ Schottky space embedded into $ X(F_g) $ and therefore into $ \C^N $ ($N$ some sufficiently large integer) by traces: this is possible
in general by standard results, and the particular traces may be chosen so that $ G \in X(F_g) $ is conjugate to a subgroup of $ \PSL(2,\R) $ if and only if all its
coordinates are real; i.e.\ the Fuchsian locus of $ \mc{S}_g $ is $ \mc{S}_g(\R) \coloneq \mc{S}_g \inter \R^N $.

\begin{ex}
  Of special interest to us is the Fuchsian locus of genus $2$ Schottky groups. These act on $ \H^2 $ to produce thrice-holed spheres; if $ G $ is such a group
  then the generators of $ G $ can be chosen to be $ X $, $ Y $, and $ XY $ such that all three of these elements are primitive boundary hyperbolics. The traces
  of these three elements of $ F_2 = F\{X,Y\} $ may be taken to parameterise $ X(F_2) $ and they are all real if and only if the corresponding representation is conjugate
  to a subgroup of $ \PSL(2,\R) $. We gave an explicit parameterisation of $ X(F_2) $ by these traces above, in \cref{ex:schottky_85}.
\end{ex}
\begin{lem}\label{lem:schottky_pleating_var}
  The locus $ \mc{S}_2(\R) $ is a union of connected components of the four semi-algebraic sets
  \begin{displaymath}\begin{array}{cc}
    (-\tr X, -\tr Y, -\tr XY) > -2&
    (-\tr X, \tr Y, \tr XY) > 2\\
    (\tr X, -\tr Y, \tr XY) > 2&
    (\tr X, \tr Y, -\tr XY) > 2.
    \end{array}
  \end{displaymath}
  The closure of each connected component meets a cusp.
\end{lem}
\begin{proof}
  The set $ \mc{S}_2 \inter \R^3 $ is a subset of the union of the four semi-algebraic sets. Consider an analytically parameterised path $ G(t) $
  through $ \mc{S}_2 \inter \R^3 $; for this path to reach the boundary, since finitely generated Fuchsian groups are geometrically finite, a new
  parabolic must appear in the group, say when $ t = 0 $; we will show that if any word becomes parabolic, then necessarily one of $ X $, $ Y $, or $ XY $ becomes
  parabolic. Suppose that $ W $ is the word representing a geodesic that becomes parabolic, so $ \tr^2 W \to 4 $ as $ t \to 0 $, and let $ \xi,\xi' \in \Lambda(G(t)) $ be the fixed points of $W$
  so that $ \xi \to \xi' $ as $ t \to 0 $. On the circle bounding the hyperbolic disc acted on by $ G(t) $, there are two arcs bounded by $ \xi $ and $ \xi'$;
  let $ S(t) $ be the arc out of these two which contracts to the empty set when $ t \to 0 $. Since $ S(t) $ necessarily contains an attractive fixed point of an
  element $ g\in G(t) $, it must contain an arc of discontinuity of $ G(t) $ for $ t \neq 0 $ (namely, an image under a sufficiently high power of $ g $ of any boundary
  component). As $ t \to 0 $, since limit points cannot pass each other on the unit circle as this would produce either parabolics or relators in the group,
  this arc of discontinuity must also tend to length $0$ as $ t \to 0 $. In particular the two fixed points of the boundary hyperbolic bounding this arc must
  collide and at $ t = 0 $ this boundary hyperbolic, which is one of $ X $, $ Y $, or $ XY $ (or their inverses), must become parabolic. In other words, the path $ G(t) $ reaches the
  boundary exactly when one of $ X $, $ Y $, or $ XY $ becomes itself parabolic.
\end{proof}

More generally, consider a Fuchsian group $ F $ uniformising a genus $ g $ surface $S$ with $ b > 0 $ boundary components and $p$ punctures. This group
can be viewed as a Kleinian group uniformising a single surface with $ p $ rank $1 $ cusps and genus $ g + b-1 $, obtained by `doubling' $ S $ across its
boundary and punctures. The Teichm\"uller space $ T(S) $ embeds into the character variety $ X(F) $ as a subset of the quasiconformal deformation
space $ \QH(F) $. By a series of ideas originating with Fricke and studied extensively by (for example) Keen~\cite{keen71,keen73}, we can choose a
generating set for $ F $, say $ f_1,\ldots,f_N $, such that the traces of all elements of $F$ are $\Z$-polynomial in the traces of all words in these generators;
more precisely, we can take the elements $ f_1,\ldots,f_N $ to be the set of all words in $ F $  which can be written in the usual surface-group generators
with word-length less than some universal bound depending on $ g $, $ b $, and $ p $. That is, we can choose an embedding of $ X(F) $ into $ \C^N $ for
some large $ N $ such that a representation $ \rho \in X(F) $ is Fuchsian if and only if the $\C^N$-coordinates of $ \rho $ are real.

Let $ \mc{T}_{g,b,p} $ be the set $ \QH(F) $ and let $ \mc{T}_{g,b,p}(\R) $ be the set $ T(S) $; for example, $ \mc{S}_g = \mc{T}_{0,g+1,0} $. A proof
of the following result may be found in Saito~\cite[\S 6]{saito94f}:
\begin{prp}\label{prp:schottky_rays}
  The locus $ \mc{T}_{g,b,p}(\R) \subset \C^N $, where $ \C^N $ is coordinatised by the traces $ \tr f_j $ as just described,
  is a semi-algebraic set. \qed
\end{prp}
In later sections of the current paper we will be restricting to maximal laminations, and in particular every F-peripheral group will be genus $2$. Hence the only
explicit inequalities which will be interest are those for the special case which we gave the details of in \cref{lem:schottky_pleating_var} above; computing them
in general is complicated (for the special case of a compact surface, i.e.\ $ b = p = 0 $, see Komori~\cite[Theorem~4.2]{komori97}) so since we do not need them in
further cases we will not work them out.

\begin{rem}
  We will end up working with real subvarieties of $ \mc{T}_{g,b,p}(\R) $ obtained by restricting the hyperbolic lengths of various
  boundary components to be equal to hyperbolic lengths of other boundary components. Of course, the hyperbolic length is not algebraic, but
  we can set up algebraic conditions to force lengths to be equal by fixing the trace of the corresponding elements to be equal. In reality, all these groups with identified boundary components
  will be embedded into the overgroup $ G $ in such a way that the elements whose traces are to be equal are conjugate in $ G $, and so the
  fact that we end up with proper subvarieties of the Teichm\"uller space will be induced by the global combinatorics and we will not need to
  impose it by hand.
\end{rem}

Returning to our general setup, let $ G $ be a convex cofinite manifold group with $ \Omega(G) \neq \emptyset $ and let $ \Lambda $ be a rational lamination
on $ \Omega(G)/G $. Define the groups $ \Pi_{i,j} \leq G $ as in previous sections, where $ i $ ranges from $ 1 $ to $ n $ (the number of components of the
boundary when rank $1$ cusps are plumbed in) and where $ j $ ranges from $ 1 $ to $ m_i $ (the number of components of the complement of $ \Lambda_i $ on
the boundary surface $ S_i $). We do not assume that the $ \Pi_{i,j} $ form a circle chain, but it is true that each of these groups is quasi-Fuchsian and
so acts on some quasidisc $ \Delta_{i,j} $ (possibly not contained within $ \Omega(G) $); let $ (g_{i,j}, b_{i,j}, p_{i,j}) $ be the respective genus, number
of boundary components, and number of punctures of $ \Delta(\Pi_{i,j})/\Pi_{i,j} $. To save space we will write $ \mc{T}^{(i,j)} $ for $ \mc{T}_{g_{i,j}, b_{i,j}, p_{i,j}} $.

\begin{defn}
  We define three product spaces:
  \begin{displaymath}
    \mathbold{X}_\Lambda \coloneq \prod_{i=1}^n \prod_{i=1}^{m_i}  X(\pi_1(P_{i,j})),\quad
    \mathbold{T}_\Lambda \coloneq  \prod_{i=1}^n \prod_{i=1}^{m_i} \mc{T}^{(i,j)},\quad
    \mathbold{T}_\Lambda(\R) \coloneq \prod_{i=1}^n \prod_{i=1}^{m_i} \mc{T}^{(i,j)}(\R).
  \end{displaymath}
  Clearly $ \mathbold{X}_\Lambda \supset \mathbold{T}_\Lambda \supset \mathbold{T}_\Lambda(\R) $. Define
  also the \df{multicharacter map}
  \begin{displaymath}
    \mathbold{\chi}_\Lambda : X(G) \to \mathbold{X}_\Lambda
  \end{displaymath}
  by taking as its components the canonical maps
  \begin{displaymath}
    \rho \ni X(G) \mapsto {\rho\!\restriction_{\Pi_{i,j}}} \circ p \in X(\pi_1(P_{i,j})).
  \end{displaymath}
  where $ p $ is the canonical map $ \pi_1(P_{i,j}) \to \Pi_{i,j} $.
\end{defn}

If $ \rho \in X(G) $ is such that $ \rho(G) $ has a $ \Lambda$-circle chain, then $ \mathbold{\chi}_\Lambda(\rho(G)) \in \mathbold{T}_\Lambda(\R) $.
By \cref{prp:schottky_rays}, this implies that $ \mc{P}_\Lambda \subset X(G) $ is identified with a holomorphic embedding (since $ \rho \mapsto \tr \rho(W) $ is a holomorphic map on $ X(G) $)
of a subset of certain branches of the inverse image under $ \mathbold{\chi}_\Lambda $ of $ \mathbold{T}_\Lambda(\R) $. We next identify the exact nature of this subset.

Firstly, one must choose the correct asymptotic branch of $ \mathbold{\chi}_\Lambda^{-1}(\mathbold{T}_\Lambda(\R)) $ to ensure that one hits the pleating ray (otherwise
one ends up with elementary groups or perfectly good Fuchsian subgroups which are not peripheral, likely embedded in an indiscrete supergroup); see for more context
the discussion in the introduction to Parker and Series~\cite{parkerseries95}. The correct branches are those whose closures meet the maximal cusp corresponding to the lamination, and
this point is can be found by solving the trace equations and then selecting the solutions which have a Ford domain giving the correct topological type.

Problems also arise once the bending angle along one of the leaves of the lamination becomes $ \pi $; this happens when the discs preserved by the two incident F-peripheral
groups merge into a single disc, so that the group generated by the union of both groups is also F-peripheral (c.f.\ \cref{ex:non_maximal}). Moving past this point
gives a concave angle between the preserved discs, the discs are no longer peripheral, and their convex hull boundaries do not project to the convex core boundary. We need the
following technical result to detect this situation.

\begin{prp}\label{cor:incidence_inequalities}
  Let $ G $ be a convex cofinite manifold group with a $\Lambda$-circle chain $ \mathbold{\Pi} = \{\Pi_{i,j}\} $. Let $ t \mapsto G_t $ ($t \in [0,1]$) be an algebraically parameterised curve
  in $ X(G) $ with $ G_0 = G $, such that for all $ t $ the images of the $ \Pi_{i,j} $ continue to be Fuchsian. There exist rational inequalities in $ t $ that detect
  the collision of the peripheral discs of the quasiconformal deformations of the F-peripheral groups in $ \mathbold{\Pi} $.
\end{prp}
\begin{proof}
  For each group $ \Pi_i(t) $ take $ u_1(t),u_2(t),u_3(t) $ to be an arbitrary choice of fixed points of three fixed elements varying with $ t $ (e.g.\ for a maximal lamination one
  might take fixed points from the elements representing the three holes of the quotient surface); the choice of one or the other of the two fixed points in the case of a hyperbolic
  element is to be taken consistently for all $ t $, so the fixed points are algebraic in $t$. By Cramer's Rule, the coefficients of the quadratic equation defining the peripheral
  disc of $ \Pi_i(t) $ are algebraic in the $ u_i(t) $, so long as the $ u_i $ are not colinear (which is always true after a suitable conjugation). Now for each pair $ \Pi_i(t) $ and $ \Pi_j(t) $
  we may take inequalities in the radius and centre of the corresponding peripheral discs (which are algebraic in the coefficients of the defining quadratic) to detect whether the
  two discs intersect or coincide.
\end{proof}

\begin{rem}
  In fact, the parameterisation is allowed to be of higher dimension. Arbitrary deformations in $ X(G) $ do not preserve
  Fuchsian groups, but if $ \Lambda $ is maximal in a Schottky group of genus $g$ (for instance) then, heuristically, the pleating variety $ \mc{P}_\Lambda $ has real dimension $ \frac{1}{2} \dim_{\R} X(G) $
  (since one is imposing realness conditions on the traces of $ 3g-3 $ variables, i.e.\ on a full set of coordinate functions for $ X(G) $).
\end{rem}

\begin{thm}\label{thm:variety_is_variety}
  The $\Lambda$-pleating variety $ \mc{P}_{\Lambda} \subset X(G) $ for a convex cofinite manifold group $ G $ and lamination $ \Lambda $ is a real semi-algebraic set.
  If a $\Lambda$-circle chain $ \{\Pi_{i,j}\} $ is chosen, then $ \mc{P}_\Lambda $ is obtained as the intersection of the feasible regions of the three sets of inequalities listed
  in (1)--(3) below, together with an additional inequality that selects the connected components of the result which meet the cusp on $ \partial \QH(G) $ indexed by $ \Lambda $
  (which may be chosen to be polynomial, because connected components of semi-algebraic sets can be separated by polynomial inequalities~\cite[Theorem~5.21]{basu}).
  \begin{enumerate}
    \item Every $ \Pi_{i,j} $ must have trace parameters satisfying the real semi-algebraic conditions in \cref{prp:schottky_rays}; that
          is, $ \mc{P}_\Lambda \subset \mathbold{\chi}_\Lambda^{-1}(\mathbold{T}_\Lambda(\R)) $.
    \item For every pair of peripheral groups in the circle chain which are incident modulo the conjugation action of $ G $ (i.e.\ have lifts which are incident in $ \FPer(G) $), an inequality
          given by \cref{cor:incidence_inequalities} must be satisfied that ensures that the two peripheral discs intersect but do not coincide.
%     \item For every pair of peripheral groups which are not incident modulo the conjugation action, an inequality given by \cref{cor:incidence_inequalities} must be satisfied that ensures that
%           the two peripheral discs are strictly disjoint.
  \end{enumerate}
\end{thm}
\begin{proof}
  Let $ S \subset X(G) $ be the set of points satisfying the claimed inequalities. If a parameter for a group lies on $ \mc{P}_\Lambda $ then clearly it satisfies the given inequalities,
  i.e. $ \mc{P}_\Lambda \subseteq S $. To see that $ \mc{P}_\Lambda $ is a union of connected components, we run an open-closed argument  similar to those in Keen and Series~\cite[\S\S3--4]{keen94}
  and Komori and Series~\cite{komoriseries98}.

  Suppose we fix a point in $ \mc{P}_\Lambda $, and then deform slightly away from it while keeping the inequality conditions satisfied. By \cref{prp:schottky_rays} every F-peripheral group remains F-peripheral
  and hence by \cref{lem:pl_ray_is_glued_discs} the new group remains on the pleating ray. Thus $\mc{P}_\Lambda$ is open in $ S $.

  To see that $\mc{P}_\Lambda$ is closed in $S \inter \QH(G) $ and thus in $ S $, suppose that $ (G_i)_{i=1}^\infty $ is any sequence of groups with a $\Lambda$-circle chain; equivalently, the corresponding
  manifolds have bending lamination $\Lambda$. By continuity of bending laminations~\cite{keenseriesCCHB} this sequence of $3$-manifolds either has bending lamination $ \Lambda $, or is a cusp group
  on $\partial \QH(G) $, or one of the leaves of $\Lambda$ has flattened
  to angle $\pi $. Thus the limit $ \lim_{i\to\infty} G_i $ either has a $\Lambda$-circle chain (so lies in $ \mc{P}_\Lambda$), or one of the dihedral angles between two peripheral discs degenerates
  to $0$ (corresponding to replacing an inequality in condition (1) of the theorem statement with an equality), or one of the angles degenerates to $\pi$ and the two discs coincide (corresponding to
  replacing an inequality in condition (2) of the theorem statement with an equality). In other words, if a sequence of groups in $ \mc{P}_\Lambda $ converges to a group in $S$, then it converges
  to a group in $ \mc{P}_\Lambda $.
\end{proof}

\section{Formal curvilinear polygons with side-pairing structures}\label{sec:formaldomains}
Our main theorems involve the algebraic movement of fundamental domains for certain Kleinian groups. We therefore require an algebraic structure which parameterises the particular class of fundamental
domains that we construct. The key observation is that we can decouple the combinatorial side-pairing structure from the geometric information associated with a fundamental domain:
the combinatorial structure is just a set of edges and a set of M\"obius transformations which pair them, and the geometric information consists of intersection data that cuts out
subspaces of possible combinatorial structures.

\subsection{Generalities on algebraic geometry of circles}
\begin{defn}[Projective spaces of circles]
  The space of circles in $ \hat{\C} $ is identified with a subset of $ \P^3 $~\cite[Vol. II,~\S 20.1]{berger}.\footnote{Throughout this paper, all projective spaces $\P^r$ are real projective spaces.}
  Having already chosen an affine coordinate system for the Riemann sphere, we define this identification by
  \begin{displaymath}
    \V(k (x^2+y^2) + (ax + by) + h) \mapsto [k:a:b:h]
  \end{displaymath}
  where $ \V $ is the usual map sending a polynomial to its corresponding affine hypersurface. The radius-squared and centre of
  a circle ($k \neq 0 $) are algebraic in its homogeneous coordinates, with formulae
  \begin{displaymath}
    \rad^2\, [k:a:b:h] = \frac{a^2+b^2 - 4kh}{4k^2} \;\text{and}\; \cen\, [k:a:b:h] = \left( -a/2k, -b/2k \right).
  \end{displaymath}
  The inner product corresponding to the quadratic form $ \rad^2 $ may be used to compute angles between circles. The formalism also includes points (as circles of radius $0$) along
  with a copy of $ \H^3 $ representing circles with purely imaginary radius~\cite[Vol. II,~20.2.5]{berger}.

  There is a unique $ \P^1 $ passing through any pair of points $ c,c' \in \P^3 $ called the \df{pencil} spanned by $ c $ and $ c' $; we denote this line by $ \langle c, c' \rangle $.
  If $ g \in \PSL(2,\C) $ does not fix $ \infty $, then we define $ \Pen(g) $ to be the pencil in $ \P^3 $ spanned
  by the two isometric circles of $ g $. This definition does not suffice if $ g $ is order $2$, but if $ g $ is elliptic then $ \Pen(g) $
  is the pencil of circles passing through the two fixed points of $ g $ and this definition includes the order $2$ case. If $ g $ is non-elliptic,
  then the fixed points of $ g $ lie in the pencil as limits of circles, and form a distinguished frame for the pencil (they are the only two points
  on $ \Pen(g) $ in the nullspace to $ \rad^2 $). If $ g $ is elliptic or parabolic then there is no canonical frame for $ \Pen(g) $.

  If $ h \in \PSL(2,\C) $ then the isometric circles of $ hgh^{-1} $ (if they exist) are \emph{not} the $ h$-translates of the isometric circles
  of $ g $. However, the pencil $ \Pen(hgh^{-1}) $ is the $ h$-translate of the pencil $ \Pen(g) $: the images under $ h $ of the isometric circles of $ g $
  give a frame of $ \Pen(hgh^{-1}) $ that is different to the frame of isometric circles of $ hgh^{-1} $.
\end{defn}

\begin{rem}
  To save space, often we will state results and give proofs only in the `generic' situation that the relevant circle configuration is finite (i.e.\ $\infty$ is
  not a point on any circle). The cases where Euclidean lines are allowed may be obtained as limiting cases, or by changing to a different affine chart.
\end{rem}

For computational purposes the following lemma is useful; it explicitly describes the action of a matrix $ g \in \PSL(2,\C) $ on $ \Pen(g) \subset \P^3$ without
passing explicitly through the representation $ \PSL(2,\C) \to \PO(3,1) $.
\begin{lem}\label{lem:action_on_circles}
  If $ g \in \PSL(2,\C) $ is loxodromic, and if coordinates are chosen on $ \Pen(g) \cong \P^1 $ so that $ 0 $ is the repelling fixed point, $ \infty $ is the attracting fixed point
  of $ g $, and $ 1 $ is the isometric circle of $ g $ that contains the repelling fixed point, then $ g $ acts on $ \Pen(g) $ by $ g(c) = \lambda c $
  where $ \lambda \in \C $ is the coordinate of the attracting isometric circle of $ g $. If $ g $ is parabolic and coordinates are chosen on $ \Pen(g) \cong \P^1 $
  so that $ \infty $ is the fixed point of $ g $ and the isometric circles of $ g $ are $ 0 $ and $ 1 $, then $ g $ acts in these coordinates by $ g(c) = c+1 $.
\end{lem}
\begin{proof}
  This is a routine calculation using the fact that $ g $ is acting as a projectivity in $ \P^3 $ restricted to a $ \P^1 $, so is acting as an element of $ \PSL(2,\R) $.
\end{proof}

\begin{rem}
  In connection with \cref{lem:action_on_circles}, it may be useful to recall that for a hyperbolic transformation represented by $ M \in \PSL(2,\C) $, the
  attracting fixed point is the projectivisation of the eigenvector of $ M $ corresponding to the eigenvalue $ \lambda $ satisfying $ \abs{\lambda} > 1 $;
  and the eigenvalues satisfy $ \lambda + \lambda^{-1} = \tr M $.
\end{rem}

A more general but less explicit version of \cref{lem:action_on_circles}:
\begin{lem}\label{lem:action_on_circles2}
  If $ g : \C^n \to \PSL(2,\C) $ is a rationally parameterised matrix, and $ \phi : \C^m \to (\operatorname{Conf}_\C(3))^k $ is a family of $k$ rationally parameterised
  triples of points, then (i) the circle through $ \phi(z) $ moves rationally in $ \P^3 $ with $ z \in \C^m $, and (ii) the images $ g(w) \circ \phi(z) $
  move rationally in terms of $ w \in \C^n $ and $ z \in \C^m $.
\end{lem}
\begin{proof}
  Part (i) is an application of Cramer's rule to the system of equations $ \{q_i(z) = 0 : 1 \leq i \leq k \} $ where each $ q_i $ is a homogeneous quadric with
  variable coefficients; Cramer's rule shows that the coefficients of the quadrics move rationally with $ z \in \C^m $. Similarly part (ii) is an
  application of Cramer's rule to the system $ \{ q_i(g(w)(z)) = 0 : 1 \leq i \leq k \} $.
\end{proof}

An important type of result for our applications is a result guaranteeing that certain incidence conditions are semi-algebraic. We have already
seen in \cref{cor:incidence_inequalities} above that incidence between circles is semi-algebraically decidable. In addition, incidence of points on a pencil
in $ \P^3 $ is algebraic:
\begin{lem}\label{lem:line_incidence}
  The set of triples $ (x,y,z) \in (\P^3)^3 $ satisfying $ x \in \langle y,z \rangle $ is an algebraic set, cut out by equations in the homogeneous coordinates
  of $ x $, $ y $, and $ z $.
\end{lem}
\begin{proof}
  Asking if $ x $ is included in $ \langle y,z \rangle $ is equivalent to asking that the system of equations $ x = \lambda y + \mu z $
  has a solution for $ (\lambda,\mu) \in \R^2 $. Let $ A $ be the $ 4\times 2 $ matrix encoding this system; then a solution
  exists if and only if $ x = \operatorname{Proj}_{\operatorname{Ran}(A)} x $ and by standard results in linear algebra this is equivalent to
  $ (A(A^tA)^{-1}A^t - I_4) x = 0 $.
\end{proof}

The problem of determining whether (convex) polytopes in $ \R^n $ intersect nontrivially is a well-studied problem due to its applications in
linear programming. The main computational difficulty is in computing the vertices from the linear inequations which cut out the polytope, or
in computing the linear inequations from the vertices. Once this computation is done, by convexity the problem is reduced to simply checking the (finitely many)
vertices of one polytope against the linear inequations of the other.

In our setting, we have the advantage that our vertices are easy to determine since we know the incidence structure of the circles cutting out our fundamental
domains. However, we no longer have convexity since edges of our polygons are allowed to be circle pieces, hence it is possible for our polygons to intersect
in the interior of edges and not at vertices: it is no longer \textit{a priori} a finite problem. We therefore need to rely on heavier machinery.

\begin{lem}\label{lem:circle_intersection}
  The intersection problem for curvilinear polygons (i.e.\ domains cut out by finitely many circle arcs and line segments) in $ \hat{\C} $ is real semi-algebraic in the homogeneous coordinates
  of the defining circles.
\end{lem}
\begin{proof}
  This follows from quantifier elimination for $\R$,~\cite[Corollary~2.75]{basu}. More precisely, suppose that our two curvilinear polygons are defined respectively by
  the intersections
  \begin{displaymath}
    P = \left\{ z:\forall_{i=1}^n\, \epsilon_i \abs{z - w_i}^2 < \epsilon_i r_i^2 \right\} \;\text{and}\;
    P' = \left\{ z:\forall_{j=1}^{n'}\, \epsilon'_j \abs{z - w'_j}^2 < \epsilon'_j {r'}_j^2 \right\}
  \end{displaymath}
  where $z$ is a complex indeterminate and where each of the $ w_i,w_j' \in \C $, $ r_i,r'_j \in \R $, and $ \epsilon_i,\epsilon'_j \in \{\pm 1\} $ for
  $ i \in \{1,\ldots,n\} $ and $ j \in \{1,\ldots,n'\} $. Then the problem is to determine the truth of the sentence
  \begin{multline*}
    \Psi = \bigexists_{(x,y) \in \R^2}\; \Biggl[ \Bigl( \ \bigwedge_{i=1}^n \epsilon_i (x - \Re w_i)^2 + \epsilon_i (y - \Im w_i)^2 < \epsilon_i r_i^2  \Bigr) \wedge{}\\
                                          \Bigl( \ \bigwedge_{j=1}^{n'} \epsilon'_j(x - \Re w'_i)^2 + \epsilon'_j(y - \Im w'_i)^2 < \epsilon'_i {r'}_i^2  \Bigr)\Biggr]
  \end{multline*}
  The content of the quantifier elimination theorem is that there is a quantifier-free sentence in the language of ordered fields (i.e.\ a sentence defining a semi-algebraic set) that
  is true over $ \R $ if and only if $ \Psi $ is true, at the expense of increasing the number of dimensions of the problem. This theorem is constructive and implemented in the
  software \texttt{QEPCAD}~\cite{collins91,qepcad} (also available as an optional package for \texttt{Sage}).
\end{proof}

\subsection{Families of purely formal fundamental domains}
We now relate these general algebraic structures to fundamental domains of Kleinian groups. For comparison the reader should refer to the definition of fundamental
domain of $ G \leq \PSL(2,\C) $ acting on $ \Omega(G) $ given in~\cite[\S II.G]{maskit} and the definition of a fundamental polyhedron for the action
of $ G $ on $ \H^3 $ given in~\cite[\S IV.F]{maskit}.

\begin{defn}
  A \df{formal curvilinear polygon with side-pairing structure} in $ \hat{\C} $ consists of:
  \begin{itemize}
    \item A choice of $ r $ lines (i.e.\ $ \P^1$'s) in $ \P^3 $, $L_1,\ldots,L_r $;
    \item For every $ i $, a choice of two (not necessarily distinct) points $ F_{i,1} $ and $ F_{i,2} $ on $ L_i $ such that $ \rad^2 F_{i,- } > 0 $ (i.e.\ both points
          give circles or lines in $ \Sph^2 $); and
    \item For every $ i $, some $ \phi_i \in \PSL(2,\C) $ so that $ L_i = \Pen(\phi_i) $ and $ \phi_i(F_{i,1}) = F_{i,2} $.
  \end{itemize}
\end{defn}

\begin{lem}
  The set of formal curvilinear polygons with side-pairing structures with $ 2r $ sides is parameterised by a semi-algebraic
  subset of $  (\P^3 \times \PSL(2,\C))^r $.
\end{lem}
\begin{proof}
  The data of a formal curvilinear polygon with side-pairing structure can be recovered from a choice $ \phi_1,\ldots,\phi_r $ of $ r $ elements of $ \PSL(2,\C) $ together with
  a choice of a single circle (of nonzero radius) or line in $ \Pen \phi_i $ for each $ i $. The set of all formal curvilinear polygons of rank $ r $ is therefore $ 7r $ dimensional
  over $ \R $, or $ 7r-6 $-dimensional up to conformal automorphisms or conjugation in $ \PSL(2,\C) $, and is parameterised by the Cartesian product $ (\P^3 \times \PSL(2,\C))^r $.

  Picking coordinates by writing an arbitrary element of the Cartesian product as $ ((x_1, \phi_1),\ldots,(x_r, \phi_r)) $, the subset of admissable points is cut out by the $ r $
  inequalities $ \rad^2 x_i > 0 $ together with the conditions $ P(i) \coloneq \text{``}x_i \in \Pen(\phi_i)\text{''} $ for all $ i \in \{1,\ldots,r\} $. We can rephrase $ P(i) $ as
  ``the homogeneous coordinates of $ x_i $ lie in the $ \P^1 $ spanned by the two fixed points of $ \phi_i $''. The computation of the fixed points is algebraic, requiring at most a
  quadratic field extension over the domain of definition of $ \phi_i $,\footnote{More precisely, add an additional dummy variable $ v_i $ for each quadratic field extension and then slice the
  resulting space by the quadratic polynomials $ v_i^2 = \cdots $; c.f.\ \cref{ex:schottky_85}.\label{foot:dummy}} and so the line $ \Pen \phi_i $ is algebraically defined. Checking whether $ x_i $
  lies on this line is then semi-algebraic by \cref{lem:line_incidence}.
\end{proof}

\begin{defn}
  An \df{algebraic family} of formal curvilinear polygons with side-pairing structures is an $\R$-algebraic map $ U \to (\P^3 \times \PSL(2,\C))^r $ where $ U \subset \R^s $ is some open semi-algebraic set.
\end{defn}

If $ G $ is discrete and geometrically finite, then a Ford domain for $ G $ gives rise naturally to a formal curvilinear polygon with side-pairing structure that, additionally,
has certain geometric conditions associated to it. Our constructions in \cref{sec:fundom} are motivated by this setting: the idea is that we first construct algebraic families
of formal curvilinear polygons, and then we find subsets of these families in which the polygons satisfy angle and incidence conditions that imply they are actually fundamental
domains in some open subset of the parameter space.

\begin{figure}
  \labellist
  \small\hair 2pt
  \pinlabel {$F_{1,1}$} at 79 73
  \pinlabel {$F_{1,2}$} at 141 49
  \pinlabel {$F_{3,1}$} [r] at 31 70
  \pinlabel {$F_{3,2}$} [l] at 194 65
  \pinlabel {$F_{2,1}$} [t] at 80 20
  \pinlabel {$F_{2,2}$} [b] at 131 111
  \pinlabel {$\phi_2$} [l] at 39 58
  \pinlabel {$\phi_3$} [t] at 127 103
  \pinlabel {$\phi_1$} [t] at 115 76
  \pinlabel {$\Fix \phi_2 = \infty = \Fix \phi_3$} [br] at 310 27
  \pinlabel {$F_{3,1}$} [br] at 337 53
  \pinlabel {$F_{1,2}$} [r] at 358 98
  \pinlabel {$L_1 = \Pen \phi_1$} [r] at 351 125
  \pinlabel {$L_3 = \Pen \phi_3$} [b] at 431 114
  \pinlabel {$F_{3,2}$} [tl] at 383 88
  \pinlabel {$F_{1,1}$} [l] at 371 55
  \pinlabel {$F_{2,2}$} [t] at 415 36
  \pinlabel {$L_2 = \Pen \phi_2$} [b] at 463 45
  \pinlabel {$F_{2,1}$} [t] at 348 27
  \pinlabel {$\P^3$} at 489 96
  \endlabellist
  \centering
  \includegraphics[width=\textwidth]{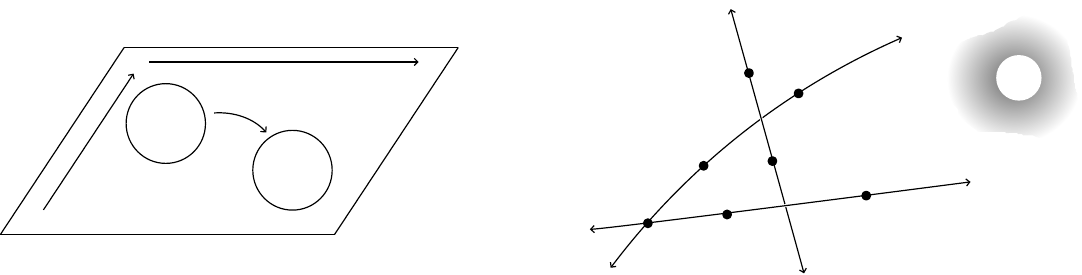}
  \caption{A formal curvilinear polygon with side-pairing structure corresponding to a $ (1;2)$-compression body. On the left, a generic fundamental domain; on the right, the structures in $ \P^3 $.\label{fig:compression_body_formal_domain}}
\end{figure}

\begin{ex}
  The generic fundamental domain for a $ (1;2)$-compression body group (i.e.\ a paralellogram with opposite sides paired by parabolics and two round discs deleted from its interior with
  boundary circles paired by a loxodromic) corresponds to a choice of
  \begin{enumerate}
    \item a line $ L_1 $ which intersects the nullspace of $ \rad^2 $ exactly twice, with two marked distinct marked points $ F_{1,1} $ and $ F_{1,2} $ so that $ \rad^2 F_{1,j} > 0 $ for each $ j \in \{1,2\}$;
    \item two projective lines $L_2 $ and $ L_3 $ in $ \P^1 $, with $ L_2 \inter L_3 = \{ \xi \} $ where $ \xi \in \P^3 $ is a point with $ \rad^2 \xi = 0 $ and where no other points on either $ L_2 $ or $ L_3 $ is
          in the nullspace of $ \rad^2 $, with two marked points on each line, $ F_{2,1}$, $F_{2,2}$, $F_{3,1}$, and $F_{3,2} $, so that $ \rad^2 F_{i,j} > 0 $ for each $ i \in\{2,3\} $ and $ j \in \{1,2\}$.
  \end{enumerate}
  This data uniquely determines $ \phi_2 $ and $ \phi_3 $ (for each $ i \in \{2,3\} $ there is a unique parabolic fixing the point in $ \hat{\C} $ represented by $ \xi $ which
  translates $ F_{i,1} $ to $ F_{i,2} $) but to determine $ \phi_1 $ an additional real number (the holonomy angle) must be given. This setup determines a fundamental domain for a discrete group
  if the two circles $ F_{1,1} $ and $ F_{1,2} $ lie inside the four-sided shape bounded by the $ F_{i,j} $ for $ i \in \{2,3\} $, $ j \in \{1,2\} $; this is a semi-algebraic condition on
  points in $ (\P^3 \times \PSL(2,\C))^3 $. \Cref{fig:compression_body_formal_domain} is a cartoon of the various geometric structures \textit{in situ}.
\end{ex}

\begin{ex}\label{ex:cone_mfd}
  In~\cite{elzenaar24c}, the central result is the construction of an algebraic family of formal curvilinear polygons that satisfy additional geometric conditions
  which are algebraic in the parameters of $ X(G) $ for $ G $ a genus $2$ surface group. These conditions imply that each of the polygons has a well-defined
  interior on $ \hat{\C} $, and that the hyperbolic convex hull of this interior is a finite-sided polyhedron in $ \H^3 $. Further, the metric space quotient
  of this polyhedron induced by the side-pairing maps is a cone manifold that is supported on the thickened genus $2$ surface and has a controlled singular
  structure.
\end{ex}

\section{Construction of fundamental domains}\label{sec:fundom}
In this section, we will continue to assume that $ G $ is a convex cofinite manifold group with $ M = \H^3/G $ and that $ \Lambda $ is a rational lamination on $ \partial M $,
and that $ X $ is a parameter space for the character variety $ X(G) $. We will give, in \cref{thm:funddom}, an explicit construction of fundamental domains for groups $ \tilde{G} \in \mc{P}_\Lambda \subset X $.
This will be an algebraic construction, in the sense that if $ \nu : U \to \mc{P}_\Lambda $ ($U$ a real ball $ \mathbb{D}(\R) $ of dimension $ \dim \mc{P}_\Lambda $) is a real algebraic
parameterisation of some subset of $ \mc{P}_\Lambda $ then the structures that define the fundamental domain of $ G_{\nu(t)} $ vary real-algebraically with $ t $ (in fact, they vary
algebraically over a finite tower of quadratic extensions of the field over which the parameterisation of representations is defined). We then extend the
construction to allow for arbitrary deformations where $ U $ is a complex ball $ \mathbb{D}(\C) $ of dimension $ \dim \QH(G) $ and prove our second main theorem, \cref{thm:mainthm}.

\subsection{Fundamental domains of F-peripheral groups}
In many of our constructions, we can restrict to the setting that $ \Lambda $ is a maximal rational lamination.
\begin{obs}\label{obs:cut_up_flat}
  Let $ \Pi $ be an F-peripheral subgroup; then there exists a finite set of F-peripheral subgroups $ \Pi_1,\ldots,\Pi_n < \Pi $
  such that (a) for all $ i $, $ \Delta(\Pi_i) = \Delta(\Pi) $, (b) for all $ i $ the surface $ \Delta(\Pi_i)/\Pi_i $ is a thrice-holed
  (or punctured) sphere, and (c) $ \hconv \Lambda(\Pi)/\Pi $ is the union of the surfaces $ \hconv \Lambda(\Pi_i)/\Pi_i $ across their
  boundaries.
\end{obs}
This follows by cutting $ \Delta(\Pi)/\Pi $ up along suitable closed geodesics and taking the corresponding subgroups as in \cref{cons:stratification_of_surfaces} above. The consequence
of the observation is that we can replace `large' peripheral subgroups that with a suitable family of thrice-holed sphere subgroups. The resulting groups do not form a circle chain, but
are a consequence of viewing everything lying on a higher-dimensional pleating variety as being a point `at the end' of a pleating variety corresponding to a maximal lamination. Because
of this we can carry out our fundamental domain constructions in the pleating rays corresponding to maximal laminations, and then take limits of these domains to get fundamental domains
for arbitrary rational laminations.

Let $ G $ admit a $\Lambda$-circle chain $\mathbold{\Pi}$. Recall we have a graph $\FPer(G)$ of F-peripheral subgroups of $ G $, which admits an action of $ G $ by conjuation;
and the circle chain $\mathbold{\Pi}$ is induced by a choice of connected fundamental domain, which is a subforest $T$ of $ \FPer(G) $, for this action. We aim to construct
fundamental domains for the action of $ G $ on $ \Omega(G) $ by locally choosing fundamental domains for each peripheral subgroup $ \Pi_{i,j} \in \mathbold{\Pi} $, and then
making small modifications so that the action of $G$ correctly glues together the fundamental domains for the peripheral subgroups on the extremities of $T$.

\begin{defn}
  Let $ \Pi $ be a genus $2$ Fuchsian Schottky group uniformising a thrice-holed sphere, and let $ f_2 $, $ f_4 $, and $ f_6 $ be primitive hyperbolic elements
  representing the three boundary components. A fundamental polygon for $ \Pi $ acting as a subgroup of $ \Isom^+(\H^2) $ is called \df{minimally bounded} if
  it has three ideal edges and six geodesic edges paired by $ f_2 $, $ f_4 $, and $ f_6 $ (see the leftmost image in \cref{fig:thrice_holed_formal_domain}).
\end{defn}
This definition stands in opposition to a classical Schottky domain for the same group, which has four ideal edges such that two of them are joined end-to-end to
form one of the three boundary holes of the quotient surface.

\begin{figure}
  \labellist
  \small\hair 2pt
  \pinlabel {$F_{1,2}$} at 87 180
  \pinlabel {$F_{3,2}$} at 150 180
  \pinlabel {$f_2$} [rb] at 89 112
  \pinlabel {$f_6$} [lb] at 151 111
  \pinlabel {$F_{1,1}$} at 18 86
  \pinlabel {$f_4$} [t] at 110 66
  \pinlabel {$F_{3,1}$} at 217 86
  \pinlabel {$F_{2,1}$} at 28 28
  \pinlabel {$F_{2,2}$} at 194 25
  \pinlabel {$\Sigma = \partial \Delta$} [t] at 117 0
  \pinlabel {$\Delta$} at 115 93
  \pinlabel {$l_2 = \Pen f_2$} [r] at 318 165
  \pinlabel {$F_{1,2}$} [tl] at 348 170
  \pinlabel {$l_1 = \langle F_{3,2}, F_{1,2} \rangle$} [l] at 392 181
  \pinlabel {$F_{3,2}$} [t] at 450 144
  \pinlabel {$\Sigma$} [r] at 420 134
  \pinlabel {$l_6 = \Pen f_6$} [l] at 497 128
  \pinlabel {$F_{3,1}$} [r] at 509 94
  \pinlabel {$l_5 = \langle F_{2,2}, F_{3,1} \rangle$} [l] at 481 46
  \pinlabel {$F_{2,2}$} [b] at 430 27
  \pinlabel {$l_4 = \Pen f_4$} [r] at 387 6
  \pinlabel {$F_{2,1}$} [bl] at 370 31
  \pinlabel {$l_3 = \langle F_{1,1}, F_{2,1} \rangle$} [t] at 300 84
  \pinlabel {$F_{1,1}$} [l] at 319 134
  \endlabellist
  \centering
  \includegraphics[width=\textwidth]{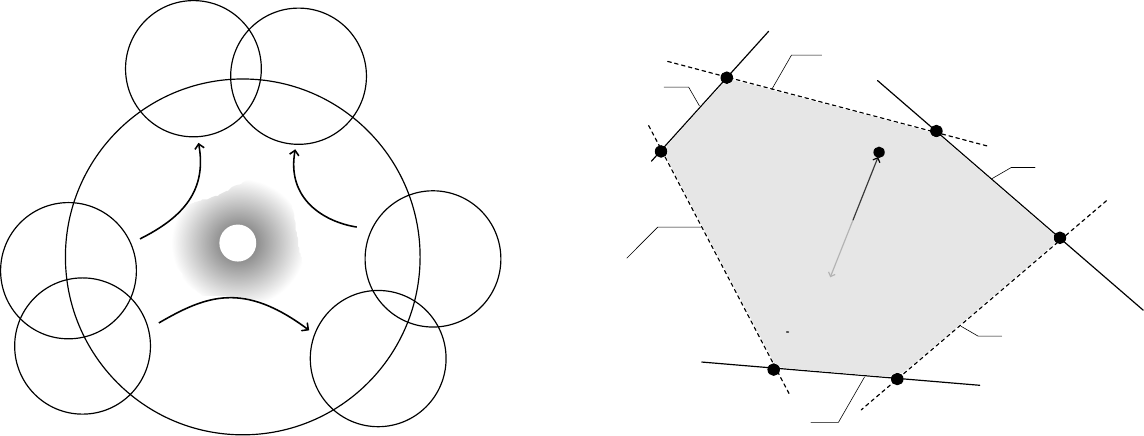}
  \caption{A formal curvilinear polygon with side-pairing structure corresponding to a Fuchsian thrice-holed sphere group, illustrating some additional geometric
          constraints. On the left, a minimally bounded fundamental domain; on the right, the structures in $ \P^3 $ corresponding to the three side-pairings (solid
          lines) as well as the three elliptic pencils (dashed lines). The six lines are not coplanar in $ \P^3 $, but every side of the hexagon is orthogonal to
          the circle $ \Sigma = \partial \Delta $ so we draw the latter as an `orthogonal vector'.\label{fig:thrice_holed_formal_domain}}
\end{figure}

\begin{rem}
  Referring to the convenient table T.1 on p.~vii of~\cite{hertrichjeromin} we can `projectivise' the definition of a minimally bounded domain.
  We say that a set of six lines in $ \P^3 $ forms a \df{skew-hexagon} if there is a cylic ordering (i.e.\ indices taken mod $6$) of the lines $ l_1,\ldots,l_6 $ such that $ l_i $
  intersects $ l_{i+1} $ and $ l_{i-1} $ transversely and does not intersect any other $ l_i $. Let $ S^2 $ be the nullspace of the quadratic form $ \rad^2 $ in $ \P^3 $,
  and let $ \Sigma \in \P^3 $ be the circle bounding the peripheral disc of $ \Pi $. Then a minimally bounded fundamental domain is a skew-hexagon $ l_1,\ldots,l_6 $ such that:
  \begin{enumerate}
    \item Each $ l_i $ is orthogonal to $ \Sigma $, that is if $ [ -,- ] $ is the Minkowski product arising from $ \rad^2 $ then $ [ l_i, \Sigma ] = 0 $.
    \item The boundary-parallel primitive elements $ f_2, f_4, f_6 $ of $ \Pi $ have the property that $ f_i(l_i) = l_{i+2} $.
    \item If $ i $ is even, then $ l_i $ intersects $ S^2 $ exactly twice (or exactly once, if the generator $ f_{i/2} $ of $ \Pi $ is parabolic).
    \item If $ i $ is odd, then $ l_i $ does not meet $ S^2 $.
  \end{enumerate}
  When $ i $ is even, $ l_i = \Pen(f_i) $. The intersection points $ l_i \inter l_{i+1} $ are the six circles in the minimally bounded domain for $ \Pi $. The corresponding
  formal curvilinear polygon with side-pairing structure is shown in \cref{fig:thrice_holed_formal_domain}.
\end{rem}

\begin{figure}
  \labellist
  \small\hair 2pt
  \pinlabel {$T$} [l] at 71 95
  \pinlabel {$S^*$} [r] at 137 122
  \pinlabel {$\tilde{T}$} [l] at 276 98
  \pinlabel {$S^*$} [r] at 347 122
  \endlabellist
  \centering
  \includegraphics[width=.9\textwidth]{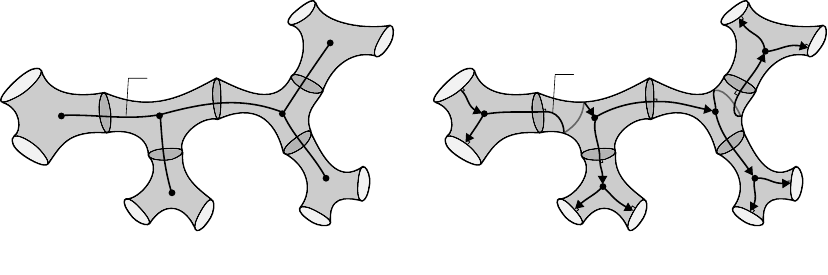}
  \caption{The surface $ S^* $ and embedded trees appearing in the proof of \cref{lem:compatible_min_bdd_domains}.\label{fig:compatible_domains}}
\end{figure}

\begin{lem}\label{lem:compatible_min_bdd_domains}
  We may choose for each $ \Pi \in \mathbold{\Pi} $ a minimally bounded fundamental polygon $ D(\Pi) $
  for the action on its corresponding peripheral disc $ \Delta(\Pi) $ in such a way that the domains are compatible: if $ \Pi $ and $ \Pi' $
  are two elements of $ \mathbold{\Pi} $ which intersect in a primitive subgroup $ \langle g \rangle $, then the facets of $ D(\Pi) $ paired by $ g $
  are arcs of the same circles as the facets of $ D(\Pi') $ paired by $ g $.
\end{lem}
\begin{proof}
  Let $ T $ be the tree in $ \Gamma(\Lambda) $ (\cref{cons:stratification_of_surfaces}) which is dual to $ \mathbold{\Pi} $. Consider the possibly disconnected pleated surface
  $ S = \partial \hconv(\Lambda(G))/G $; its bending lamination is $ \Lambda $. Cut this surface along the leaves of $ \Lambda $ which do not intersect the tree $ T $. The result is
  an abstract pleated hyperbolic surface $S^*$ with a number of geodesic boundary components, as in the left cartoon of \cref{fig:compatible_domains}. Let $ \tilde{T} $ be the tree obtained
  by adding a leaf to $ T $ for every boundary component; the resulting tree has only leaves and trivalent vertices. Choose an orientation on $ \tilde{T} $, i.e.\ a choice of
  orientation for each edge, such that a vertex has no out-edges only if it is a leaf of $ \tilde{T} $, and such that there is exactly one leaf with an out-edge. Arbitrarily fix a point on the boundary
  component corresponding to this unique `source' leaf, and fire from this point an orthogonal geodesic $ \xi $ into the interior of the thrice-holed sphere. Let $ \beta_1 $ and $ \beta_2 $
  be the boundary components corresponding to the out-vertices from the vertex of $ \tilde{T} $ corresponding to this thrice-holed sphere. Choose an arbitrary
  point on $ \xi $; for each $ i \in \{1,2\} $ there is a unique geodesic ray from $ \xi $ that meets $ \beta_i $ orthogonally. Continuing via induction on trivalent vertices of $ \tilde{T} $
  ordered according to the orientation, we obtain a geodesic embedding of $ \tilde{T} $ onto $ S^* $ so that leaf vertices lie on boundary components, trivalent vertices lie in the interior
  of flat pieces, and every edge of $ \tilde{T} $ is orthogonal to all of the leaves of the bending locus of $ S^* $ that it intersects. See the right cartoon of \cref{fig:compatible_domains}.

  Lift this topological graph to $ \partial \hconv(\Lambda(G)) $. The lift of an edge $e$, $ \tilde{e} $, is piecewise geodesic in $ \H^3 $, with geodesic pieces on
  each `dome' of the hyperbolic convex hull boundary and possible corners only at intersections between adjacent domes. By construction, each component of $ \tilde{e} $ lies on two geodesic
  domes which are (up to the action of $G$) of the form $ \hconv \partial\Delta{\Pi_{i,j}} $ and crosses exactly one pleating curve, which is the axis of the primitive hyperbolic or parabolic
  element $ g $ lying in the intersection of the corresponding F-peripheral groups in the circle chain. Since the intersection of $ \tilde{e} $ with $ \Ax(g) $ is orthogonal (by construction),
  each connected component of $ \tilde{e} $ is a connected subset of the intersection of $ \hconv{\Lambda(G)} $ with a geodesic dome supported on a circle in $ \Pen(g) $.
  In other words, cutting $ S^* $ along the topological embedding of $ \tilde{T} $, lifting to $ \partial \hconv(\Lambda(G))/G $, and projecting in the natural way to $ \hat{\C} $, gives a
  system of minimally bounded fundamental polygons for the circle chain $ \mathbold{\Pi} $, with all paired edges lying in the pencil of the side-pairing transformation, and with all edges
  matching across the `internal' boundaries of the circle chain as desired.
\end{proof}

\begin{figure}
  \centering
  \begin{subfigure}[t]{.49\textwidth}
    \centering
    \includegraphics[width=\textwidth]{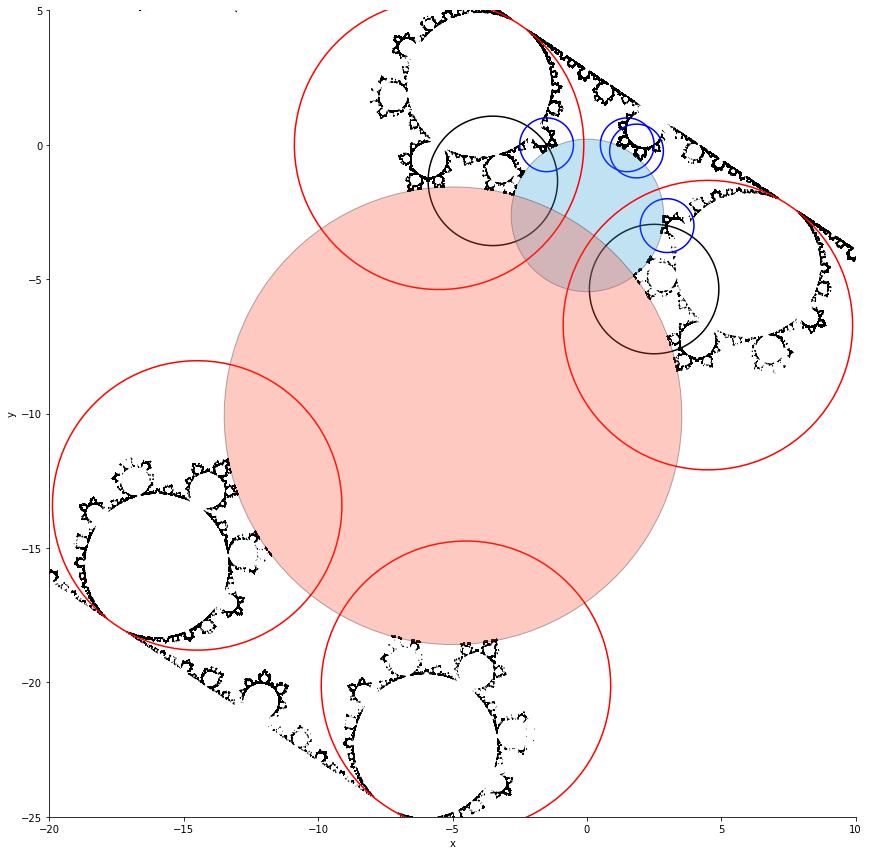}
    \caption{Isometric circles of the three boundary-parallel elements for each of the groups.}
  \end{subfigure}\hfill
  \begin{subfigure}[t]{.49\textwidth}
    \centering
    \includegraphics[width=\textwidth]{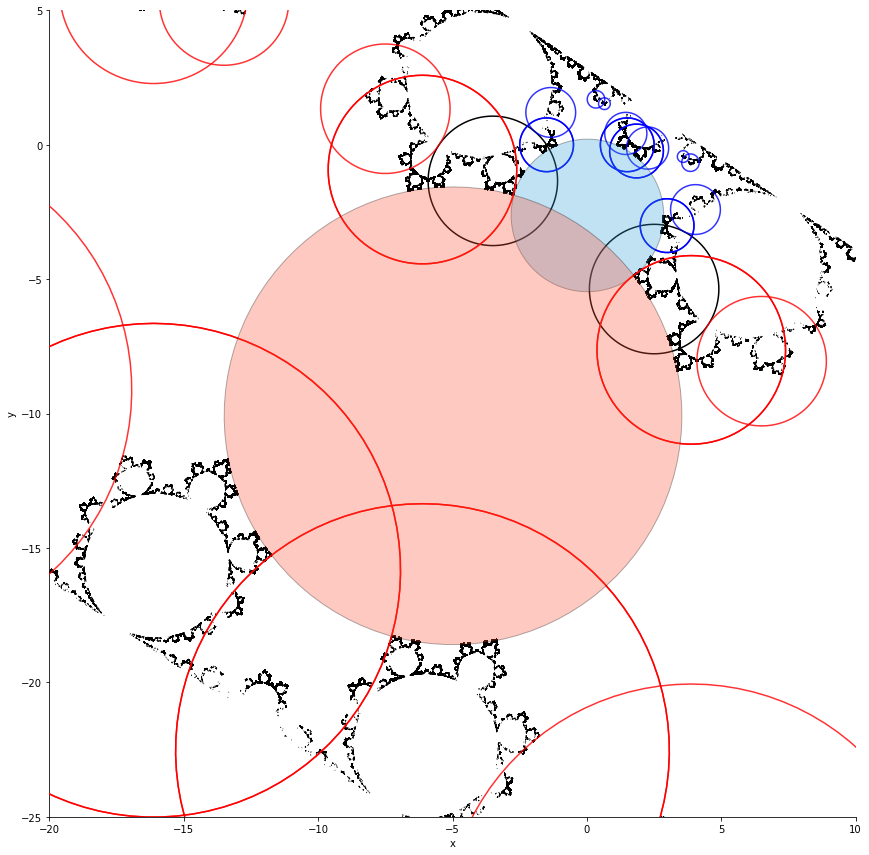}
    \caption{Circles forming compatible minimally bounded fundamental domains for the two groups.\label{fig:function2}}
  \end{subfigure}
  \caption{Peripheral subgroups of a partially cusped genus two function group. A choice of peripheral discs is shaded: the smaller blue disc is $ \Delta(\Pi_{1,1}) $ and
  the larger red disc is $ \Delta(\Pi_{1,2}) $.\label{fig:function}}
\end{figure}

\begin{ex}\label{ex:function}
  Consider the space $X$ of function groups~\cite[Chapter~X]{maskit} obtained from a product $ S_{2,0} \times I $ by pinching one end to a pair of thrice-punctured
  spheres to produce three accidental parabolics on the other end. Then $ X $ is a subvariety of the character variety of the genus $2$ surface
  group $ S_{2,0} $. Write $ \pi_1(S_{2,0}) = \langle P,Q,M,N : [P,Q][M,N] \rangle $; in~\cite[Proposition~3]{elzenaar24c} we gave an explicit parameterisation of $X$
  by $ (\alpha,\beta,\sigma) \in \C^3 $. The loops on $ S_{2,0} $ represented by the three words $ M $, $ M^{-1}Q $, and $ M^{-1}N^{-1}P M P^{-1} $
  form a maximal rational lamination $\Lambda$; one choice of fundamental groups for the two components of $ S_{2,0} \setminus \Lambda $ is
  \begin{displaymath}
    \Pi_{1,1} = \langle M, P^{-1}NMN^{-1}P \rangle\;\text{and}\; \Pi_{1,2} = \langle M^{-1}Q,P^{-1}M^{-1}Q P \rangle.
  \end{displaymath}
  With some experimentation we find that a compression body group where both of these groups are F-peripheral and where the three words
  representing $\Lambda$ are hyperbolic with trace $ \pm 3 $ corresponds to the $X$-coordinates
  \begin{displaymath}
    \alpha = 10-3i \sqrt{5},\; \beta = \frac{5}{29}(171+2i \sqrt{5}),\;\text{and}\;\sigma = \frac{4}{29}(3-2i \sqrt{5}).
  \end{displaymath}
  The limit sets of this group and the peripheral discs of $ \Pi_{1,1} $ and $ \Pi_{1,2} $ are shown in in \cref{fig:function}. We also show the isometric circles of
  the generators; the primitive hyperbolic word in the intersection $ \Pi_{1,1} \inter \Pi_{1,2} $ is $ M^{-1}P^{-1}NMN^{-1}P $. The Ford domain of $ \Pi_{1,1} $
  is already minimally bounded. Using \cref{lem:action_on_circles}, we can pick circles in the pencils of the two generators of $ \Pi_{1,2} $ which intersect
  the isometric circles of the intersection word, and compute the circles that they are paired with. Thus we can produce a pair of compatible minimally bounded
  fundamental domains, shown in \cref{fig:function2}. We also show the translates of the circles forming the domain to the four adjacent
  peripheral discs (the translating elements are the elements used in \cref{cons:stratification_of_surfaces} to produce an HNN extension: in this case, $ P^{\pm 1} $ and $(NP^{-1})^{\pm 1} $).
\end{ex}

\subsection{Global fundamental domains from circle chains}\label{sec:fundom_global}
We continue with the same notation as in the statement and proof of \cref{lem:compatible_min_bdd_domains}; so $G$ is a group on a pleating ray with
a fixed circle chain $ \mathbold{\Pi} $ and we have chosen minimally bounded fundamental domains for each F-peripheral group in $ \mathbold{\Pi} $;
the side-pairings on these domains induce a topological quotient on the union of the minimally bounded domains to form a surface $ S^* $ with
boundary components. These boundary components are paired up by the images in $ G $ of the stable elements of the various HNN extensions used
to produce the fundamental groups $ \pi_1(S_i) $ from the thrice-holed sphere groups $ \pi_1(P_i) $ (c.f.\ \cref{cons:stratification_of_surfaces}).

Take one such stable element $ g $; it pairs two boundary components by conjugation in $G$, so there are two elements $ h_1 $ and $ h_2 $ which
represent the boundary components of two F-peripheral groups in $ \mathbold{\Pi} $ so that $ h_2 = gh_1 g^{-1} $. Let $ \xi_1, \xi_1' $ be the
two circles constructed in \cref{lem:compatible_min_bdd_domains} which are paired by $ h_1 $, and let $ \xi_2,\xi_2' $ be the two circles paired
by $ h_2 $. In general it is not the case that $ g(\xi_1) = \xi_2 $ and $ g(\xi_1') = \xi'_2 $ and there is usually a `dogleg' in the topological
embedding $ \tilde{T} $ at the curves which are glued to turn $ S^* $ into $ S $. If we view the construction of the lemma as producing a formal
curvilinear polygon with edges paired by the primitive boundary hyperbolics and parabolics of the groups in $ \mathbold{\Pi} $, then
$ g $ maps the line $ L_1 = \Pen(h_1) \subset \P^3 $ to the line $ L_2 = \Pen(h_2) $ but does not send the marking $ (F_{1,1}, F_{1,2}) $
to the marking $ (F_{2,1}, F_{2,2}) $.
\begin{defn}\label{defn:holonomy}
  The unique hyperbolic M\"obius transformation with the same fixed points and axis as $ h_2 $ which sends $ (\phi(F_{1,1}), \phi(F_{1,2})) $ to $ (F_{2,1}, F_{2,2}) $
  will be called the \df{fine holonomy} of the stable element $ g $ with respect to the particular choice of compatible minimally bounded fundamental
  domains. There is a unique $ \delta\in\Z $ such that $ h_2^\delta (F_{2,1}) \between \phi(F_{1,1}) \between h_2^\delta(F_{2,2}) \between \phi(F_{2,2}) $ (where $ \between $ is the
  `betweenness' relation and can be taken to be an inequality on $ \R $ after choosing a frame) and we call $\delta$ the \df{coarse holonomy} or just \df{holonomy}
  of $ g $ in $ \langle h_2 \rangle $; $ h_2^\delta$ is the `best approximation' in $ \langle h_2 \rangle $ to the fine holonomy map.
\end{defn}
This definition is encoding a similar kind of idea to the shearing parameter studied by Parker and Series~\cite{parkerseries95}, except here we are not able to choose a
canonical curve on the surface from one side of the cut to the other that we can measure the complex distance between the start and endpoints of.

A choice of polygon where the fine holonomy associated to all of the HNN extensions is the identity map corresponds to a geometric dual graph for the bending
lamination (i.e.\ a graph which is geodesic on each flat piece and meets every pleating leaf orthogonally). It is not clear whether or not such a graph exists
in general, but in sufficiently symmetric settings it does. This includes \cref{ex:function}, where the pencils of the side-pairing maps in each peripheral disc are symmetric
with respect to the HNN extension stable elements.

Even in the case of nontrivial holonomies, we can use compatible minimally bounded domains to construct a fundamental domain for the action of $G$ on $ \Omega(G) $.
\begin{thm}[\cref{mainthm1}]\label{thm:funddom}
  Suppose that $ G $ is a convex cofinite manifold group admitting a $\Lambda$-circle chain and let $ X $ be a parameter space for $ X(G) $. There exists
  a decomposition of $ \mc{P}_\Lambda \subset X $ into countably many semi-algebraic subsets $ \mc{U}_{j} $ so that for each $ j $ there is a
  family of formal curvilinear polygons, parameterised $\R$-algebraically by $ \tilde{G} \in \mc{U}_j $, with side-pairing transformations generating $ \tilde{G} $;
  each polygon in this family induces a fundamental domain for the action of $ \tilde{G} $ on $ \Omega(\tilde{G}) $.
\end{thm}
\begin{proof}
  We first show that we can extend the construction of \cref{lem:compatible_min_bdd_domains} to produce a fundamental domain for $ G $. Fix a $\Lambda$-circle chain
  $ \mathbold{\Pi} = \{ \Pi_{i,1},\ldots,\Pi_{i,m_i} \}_{i=1}^n $ for $ G $. As in the proof of \cref{lem:compatible_min_bdd_domains} we consider the surfaces $ S^*$ obtained by cutting $ \Omega(G)/G $
  along the leaves of $ \Lambda $ that are not dual to the spanning tree of $ \Gamma(G) $ arising from the circle chain. Let $ g_1,\ldots,g_n \in G $ be choices of images of elements
  for the HNN extensions that glue the geodesic boundaries of these surfaces together (in other words, the $ g_i $ are the elements of $ G $ that generate the action of $ G $ on the graph $ \FPer(G) $ from
  the fundamental domain given by the $ \mathbold{\Pi} $). Suppose that the boundary hyperbolics of the various F-peripheral groups are $ h_1, \ldots, h_s $
  and $ g_1h_1 g_1^{-1}, \ldots, g_s h_s g_s^{-1} $. For each $ i \in \{1,\ldots,s\} $, we let:
  \begin{itemize}
    \item $ \delta_i $ be the coarse holonomy of $ g_i $ in $ \langle h_i \rangle $;
    \item $ \Pi_i $ be the F-peripheral subgroup in $ \mathbold{\Pi} $ that contains $ h_i $;
    \item $ \hat{C}(h_i) $ be a circle passing through the fixed points of $ h_i $, and
    \item $ C(h_i) = \hat{C}(h_i) \inter \Delta(\Pi_i) $.
  \end{itemize}
  For example, one choice of $ \hat{C}(h_i) $ gives $ C(h_i) = \Ax(h_i) $ (the translation axis of $ h_i $ as it acts as a hyperbolic isometry on $ \Delta(\Pi_i) $);
  alternatively one might take $ C(h_i) $ to be the Euclidean segment joining the two fixed points of $ h_i $.

  We use these data to produce a formal curvilinear polygon with side-pairing
  structure. First, take the circles used to define the compatible minimally bounded domain produced in \cref{lem:compatible_min_bdd_domains}: for each F-peripheral group in $ \mathbold{\Pi} $
  we have three pairs of circles paired by a primitive boundary parallel element which match up along intersections of peripheral discs. To these, we must add sides that will be paired
  up to glue together the boundary components of $ S^* $ to form $ S=\Omega(G)/G $. As we continue with this construction, the discussion may be clarified with reference to \cref{fig:dogleg}.

  \begin{figure}
    \labellist
    \small\hair 2pt
    \pinlabel {\color{red}$h_i$} [tr] at 31 96
    \pinlabel {$g_i h_i^{\delta_i}$} [r] at 36 67
    \pinlabel {$g_i h_i^{\delta_i+1}$} [r] at 42 40
    \pinlabel {\color{blue}$g_i h_i g_i^{-1}$} [r] at 71 22
    \pinlabel {$g_i h_i^{\delta_i+1}$} [l] at 281 132
    \pinlabel {$g_i h_i^{\delta_i}$} [l] at 274 56
    \endlabellist
    \centering
    \includegraphics[width=\textwidth]{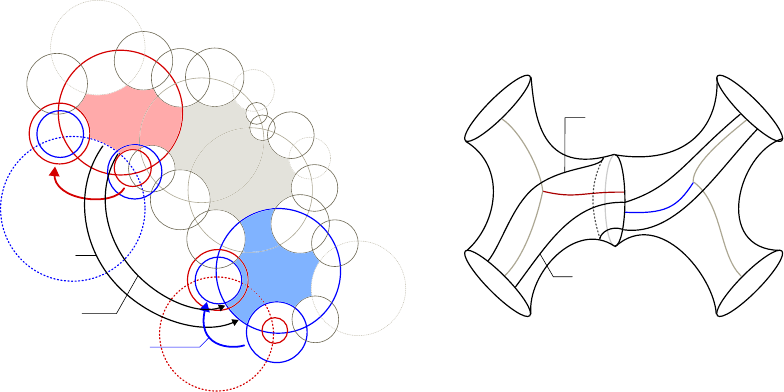}
    \caption{A dogleg arising from choosing a realisation of $ \tilde{T} $ that has holonomy. There are two valid choices for the stable element realising the gluing of the two thrice-holed
    spheres, which differ by a Dehn twist. Left: the lift of a portion of a circle chain to $ \hat{\C} $, along with the side-pairing transformations that are conjugated by $ g_i h_i^{\delta_i} $
    (here, $ \delta_i = 0 $). Right: a section of the pleated surface facing $ S_i $ that includes the pleat represented by $ g_i $; the embedding
    of $ \tilde{T} $ into $ S_i^* $ does not match across the join. \label{fig:dogleg}}
  \end{figure}

  The sides paired by the stable elements will be supported on the circles $ \hat{C}(h_i) $ and $ g_i\hat{C}(h_i) $; it remains to choose suitable side-pairing transformations here.
  Consider the circular arc $ C(h_i) $; this arc meets the two circles paired by $ h_i $. In addition, the two circles paired by $ g_i h_i g_i^{-1} $
  are mapped by $ (g_i h_i^{\delta_i})^{-1} $ to two circles that meet $ C(h_i) $. This gives four circles in total that meet $ C(h_i) $, and these circles cut out three subarcs of the arc $ C(h_i) $.
  Similarly, we obtain three subarcs of the arc $ g_i C(h_i) $ cut out by the $g_i$-translates of the circles paired by $ h_i $, together with the circles paired by $ g_i h_i g_i^{-1} $. The central
  subarc of $ C(h_i) $ is sent to the central subarc of $ g_iC(h_i) $ by $ g_i h_i^{\delta_i} $. If $ \Pi_i \in \mathbold{\Pi} $ is the F-peripheral group containing $ h_i $ then a second
  of the three subarcs of $ C(h_i) $ is contained within the fundamental domain of $ \Pi_i $. Similarly if $ \Pi_i' \in \mathbold{\Pi} $ is the F-peripheral group containing $ g_i h_i gi^{-1} $
  then there is a second subarc of $ g_i C(h_i) $ contained within the fundamental domain of $ \Pi_i' $, and these two subarcs are paired by $ g_i h_i^{\delta_i + 1} $.

  In summary, we obtain a formal curvilinear polygon with side-pairing structure consisting of the circles and side-pairings from the compatible minimally bounded domains of the circle chain groups, together
  with the circles $ \hat{C}(h_i) $ and $ g_i\hat{C}(h_i) $, which each appear twice: once paired by $ g_i h_i^{\delta_i} $ and once by $ g_i h_i^{\delta_i + 1} $. This structure
  forms a fundamental domain when an additional vertex is added in the interior of each $ C(h_i) $ as just described, with a different side-pairing on each half.

  We now show that this definition can be chosen to vary algebraically as we deform $ G $ to $ \tilde{G} $ along $ \mc{P}_\Lambda $. To do this we first make a canonical choice
  for $ \hat{C}(h_i) $, namely we take $ \hat{C}(h_i) = \partial g_i^{-1}\Delta(\Pi_i') $ where $ \Pi_i' $ is the peripheral group in $ \mathbold{\Pi} $ which contains $ g_i h_i g_i^{-1} $
  (this is not $ g_i \Pi_i $, which does not lie in $ \mathbold{\Pi} $); the location of this circle clearly depends algebraically on the trace parameters of $ G $.

  The only other arbitrary choices involved in the construction are those in \cref{lem:compatible_min_bdd_domains}: an arbitrary choice of first base point, and an arbitrary choice of point on each
  fired geodesic in the interior of a flat piece. We can choose the latter algebraically in terms of trace parameters: take a group $ \Pi \in \mathbold{\Pi} $ with the three primitive
  boundary-parallel elements $ A $, $ B $, and $ C $, and suppose that we have already chosen circles paired by $ A $. Draw the geodesic from $ \Fix^- A $ to $ \Fix^- B $. This geodesic intersects
  the circle paired by $ A $ which contains $ \Fix^- A $; this intersection point is within the hexagon bounded by $ \Ax(A) $, $ \Ax(B) $, and $ \Ax(C) $. Lifting to the hyperbolic convex hull
  boundary and projecting to the pleated surface, the intersection point lies in the interior of the flat piece uniformised by $ \Pi $. In addition, it is defined algebraically in terms of trace data
  (up to quadratic extensions to compute fixed points).

  We now show that the arbitrary first choice of base point can be made in a way that varies algebraically with trace data. Denote the first element whose paired circles need to be defined
  by $ g $ (i.e.\ $ g $ is the boundary parallel element of an F-peripheral group which corresponds to the unique leaf in $ \tilde{T} $ with edge oriented outwards). Pick the circles paired
  by $ g $ arbitrarily for the basepoint $ G $, say $ C_G $ and $ g C_G $. In terms of the frame $ \{ \Fix^-(g), I^-(g), I^{+}(g) \} $ for $ \Pen(g) $ set up by the conformal data of $ g $, the
  point $ C_G $ is determined by a single real number: namely, the number $ x $ where $ C_G = x \Fix^-(g) + I^{-}(g) $ in the homogeneous coordinates defined by $ \Fix^{-}(g) = [1:0] $,
  $ I^+(g) = [0:1] $, and $ I^-(g) = [1:1] $). If $ G $ deforms algebraically to $ \tilde{G} $, then the isometric circles of $ g $ deform algebraically to isometric circles of $ \tilde{g} $ where $ \tilde{g} $
  is the image of $ g $ under the representation $ G \to \tilde{G} $). We choose the circles paired by $ \tilde{g} $ to be $ C_{\tilde{G}} = x I^+(\tilde{g}) + I^{-}(\tilde{g}) $ and $ \tilde{g} C_{\tilde{G}} $.
  This choice clearly moves algebraically with $ \tilde{G} $, so we have shown that we can define the edges of our fundamental domains to move algebraically on $ \mc{P}_\Lambda $.

  It is possible for the coarse holonomy of the compatible minimally bounded domains to change as we move along the pleating ray; the places where this occurs
  are detected by inequalities on the the positions of the $ g_i$-translates of the circles $ \hat{C}(h_i) $, each set corresponding to a particular choice of $ \delta_i $ needed
  to make the fundamental domains at each end of the tree $ \tilde{T} $ meet along a nonempty arc after $ g_i$-translation. Within each of these semi-algebraic sets the side-pairing
  transformations between $\hat{C}(h_i)$ and $ g_i\hat{C}(h_i) $ move algebraically and so these sets form the countably many subsets $ \mc{U}_j $, cut out by geometric inequalities
  that encode the property that the corresponding formal curvilinear polygon with side-pairing structure actually forms a fundamental domain.
\end{proof}

\begin{cor}\label{cor:h3map}
  Suppose that $ G $ is a convex cofinite manifold group admitting a $\Lambda$-circle chain and let $ X $ be a parameter space for $ X(G) $. There exists
  a decomposition of $ \mc{P}_\Lambda \subset X $ into countably many semi-algebraic subsets $ \mc{U}_{j} $ so that, for each $ j $,
  there is a map from $ \mc{U}_{j} $ into the space of finite-sided polyhedra in $ \H^3 $ which sends $ \tilde{G} \in \mc{U}_j $
  to a fundamental polyhedron for $ \tilde{G} $.
\end{cor}
\begin{proof}
  By taking hyperbolic convex hulls, \cref{thm:funddom} furnishes us with a fundamental domain $D$ for the action of $ \tilde{G} $ on
  \begin{displaymath}
    \Omega(\tilde{G}) \union (\H^3 \setminus \hconv\Lambda(\tilde{G})),
  \end{displaymath}
  i.e.\ on a lift of the exterior of the convex core of $ \H^3/\tilde{G} $; this fundamental domain detects the real ends of the manifold and so is
  enough to pin down $ \tilde{G} $ using \cref{thm:mardentukia}. It can be improved to give a fundamental domain for the total action of $ \tilde{G} $
  on $ \H^3 $ by a Voronoi diagram construction generalising the Dirichlet construction~\cite[\S IV.G]{maskit}. Let $ \hat{D} $ be
  the hyperbolic convex hull of $ D $, and let $ \mc{D} $ be the set of $\tilde{G}$-translates of $ \hat{D} $. Since $ \Stab_{\tilde{G}} D $ is
  trivial, we can enumerate $ \mc{D} $ irredundantly as $ \{ g\hat{D} : g \in \tilde{G} \} $. Intersecting these sets with $ \partial \hconv \Lambda(G) $
  gives a tiling of a topological sphere in $ \overline{\H^3} $ where each tile $ g\hat{D} \inter \partial \hconv \Lambda(G) $ is made up of finitely many
  flat hyperbolic polygons (which we will call \df{cells}) glued along angles.

  For all $ g \in \tilde{G} $ we define a Voronoi region
  \begin{displaymath}
    V_g \coloneq  \left\{ x \in \left(\H^3 \setminus \cup \mc{D}\right) : \forall_{g' \in \tilde{G}}\ \rho(x, g\hat{D}) \leq \rho(x, g'\hat{D}) \right\}.
  \end{displaymath}
  where $ \rho $ is the hyperbolic metric. The set $ \hat{D} \union V_1 $ is a fundamental set for the action of $ \tilde{G} $ on $ \H^3 $.

  We now show that $ V_1 $ is cut out by finitely many planes in $ \H^3 $, i.e.\ it is a hyperbolic polyhedron. Define for each cell on the hyperbolic
  convex hull boundary a set of pleated geodesics that are hyperbolic bisectors of the boundary pleats. In general these will not meet at a unique
  point in the cell, but they will meet at a finite set of points within each cell. Since $ \tilde{G} $ acts discontinuously on $ \H^3 $ and the
  convex core boundary is embedded in $ \H^3 $, the set of all of these points over all cells is discrete. Use these points as seeds for a classical
  Voronoi diagram in $ \H^3 $; then $ V_1 $ is the union of each of the Voronoi regions with seeds in the cell above $ \hat{D} $. In particular,
  it is cut out by planes. Since $ V_1 \union \hat{D} $ is cut out by planes it is a fundamental polyhedron; but every fundamental polyhedron for $ \tilde{G} $
  is finitely sided since $\tilde{G}$ is geometrically finite~\cite[\S VI.E.2]{maskit} and so we have the desired result.
\end{proof}

\begin{rem}\label{rem:h3map}
  The map of \cref{cor:h3map} is not necessarily algebraic even within each $ \mc{U}_j $, since small movements of the parameter will in general cause new facets
  of the Voronoi cells in $ \H^3 $ to pop into existence. We only have algebraic control over the conformal boundaries of the domains involved. This is a manifestation
  of the remarks in \cref{sec:locglo} on the relationship between the group $ G $ globally and its pieces $ \Sigma_i $ `near' the conformal surfaces; we know what
  the $ \Sigma_i $ actions look like (the conformal action) but we do not know how they glue together in $ \H^3 $ without special information in particular cases.
  The maps will still be continuous in the topology induced by the Hausdorff metric on closed sets~\cite[Vol. I, \S9.11]{berger} (more precisely, cut the polyhedra
  up into two pieces along the convex core boundary; the boundary neighbourhood piece exterior to the convex core moves continuously by \cref{thm:funddom}, and the convex piece moves continuously
  by the continuity of Voronoi regions with respect to the Hausdorff metric~\cite[\S 9.1.2]{aurenhammer}).
\end{rem}

\begin{ex}[The case of $k$-punctured sphere groups with a single end]\label{ex:punc_sph_2}
  Specialise to the case that $ M = \H^3/G $ has a single conformal end, $ S $. Suppose that the conformal structure on $ S $ is a
  $k$-punctured sphere, so the lamination $ \Lambda $ consists of $ k $ loops of length $0$ together, completed to a maximal
  lamination arbitrarily; then the dual graph to the non-zero-length loops is a tree. Use this dual graph to produce the
  surface $ S^* $, which is again a $k$-punctured sphere.

  Suppose that $ h_1$ and $ h_2 $ are elements corresponding to boundary components of $ S^* $, so both are parabolic. In this special
  case, the arcs $ C(h_i) $ can be chosen to be degenerate, consisting of exactly the fixed point of the parabolic. Thus the fundamental
  domain described in \cref{thm:funddom} does not have any edges beyond those coming from the F-peripheral groups of $G $,
  and in particular there is no interaction of holonomy with the fundamental domain: we do not need to replace the elements $ h_1 $ and
  $ h_2 $ with twisted elements. Another way of seeing this is to observe that the conjugating element does not represent a loop on the
  surface $ S $, and so \emph{cannot} be visible from the image of the fundamental group of $ S $ in $ G $.

  Since $ \Lambda $ was arbitrary, this shows that all convex cofinite manifold groups $ G $ on any $ \Lambda$-pleating variety
  such that $ \Omega(G)/G $ is a $k$-punctured sphere can be given a combinatorially consistent fundamental domain that depends
  algebraically on the parameterisation of the pleating variety. This includes, as a special case, \cref{ex:punc_sph_1}.
\end{ex}

\subsection{Deformations off pleating varieties via twisting}\label{sec:twisting}
We wish to extend the map of \cref{thm:funddom} to give fundamental domains of groups in open neighbourhoods of the sets $ \mc{U}_j $. The na\"ive picture
is that we replace F-peripheral groups with quasi-Fuchsian peripheral groups, and consider pullbacks of $ \mathbold{S}_\Lambda $ under $ \mathbold{\chi}_\Lambda $
instead of only pullbacks of $ \mathbold{S}_\Lambda(\R) $. The problem is that it is not possible to semi-algebraically detect when a `quasi-circle chain' ceases to model the convex
core boundary, since there is no analogue of \cref{cor:incidence_inequalities}.

To get around this problem, we work directly with algebraic families of formal curvilinear polygons with side-pairing structures. Given a small motion
of the parameters in $ X $ off a pleating variety, we define (purely formally) a small deformation of the corresponding fundamental domain as constructed in the
previous section in such a way that the result is still a fundamental domain within a computable semi-algebraic set in $X$ (defined in terms of incidence geometry in the
formal curvilinear polygon). The procedure is analogous to a construction in our earlier joint work on the Riley slice with Martin and Schillewaert~\cite[\S 5.3]{ems21},
but has additional difficulties which could be ignored in that paper where only a single class of well-understood groups was studied.

We will continue with the notation we have been using throughout the paper: let $ G \leq \PSL(2,\C) $ be a convex cofinite manifold group which admits a $\Lambda$-circle
chain $\mathbold{\Pi} $ for some maximal rational lamination $ \Lambda $. Let $ T $ be the corresponding maximal forest in $ \Gamma(\Lambda) $, whose vertices and edges
we associate with elements of $ \mathbold{\Pi} $ and their intersections; and let $ S^* $ be the (possibly disconnected) subsurface of $ S = \bigcup_{i=1}^n S_i = \Omega(G)/G $
corresponding to the deletion of all leaves of $ \Lambda $ that do not intersect an arc of $ T $. In addition, let $ G(t) $ be a family of representations in $ \Hom(G,\PSL(2,\C)) $
that depends algebraically on $ t \in X $ for some parameter space $ X \subseteq \C^N $, so that $ G(0) = G $.

\begin{cons}\label{cons:twisting}
  Let $ D_0 $ be a choice of grafted fundamental domain (\cref{thm:funddom}) for $ G(0) $: it is a formal curvilinear polygon with side-pairing structure,
  which also happens to satisfy sufficient incidence conditions to guarantee that it is a fundamental domain for the action of $ G(0) $ on $ \hat{\C} $.
  Let $ \mc{U} \subseteq \mc{P}_\Lambda \subset X $ be the semi-algebraic set defined in \cref{thm:funddom} which contains $ G(0) $ and within which we have
  a family $ \{D_t : t \in \mc{U} \} $ of algebraically varying fundamental domains for $ G(t) $. We will construct an extension of this family indexed by $ X $ rather
  than $ \mc{U} $.

  \begin{subcons}[Definitions for the side-pairing maps]
  Each edge in $ T $ corresponds to a single cyclic group with primitive generator in one of the groups $ \pi_1(S_i) $ representing a leaf of $ \Lambda_i $. This
  gives a family of elements $ \{ f_e(0) : e \in E(T) \} $ where each element is a primitive hyperbolic or parabolic element of $ G(0) $;
  all the $ f_e(0) $ are the images in $ G $ of pairwise non-conjugate elements in the surface fundamental groups. For general $ t \in X $, define $ f_e(t) $ for $ e \in E(T) $
  to be the image of $ f_e(0) $ under the canonical map $ G \to G(t) $.

  For each boundary component or non-paired puncture of $ S^* $ (there are two for each leaf of $ T $), we can choose a corresponding primitive
  hyperbolic or parabolic element of $ G $. These primitive elements come in $G$-projections of conjugate pairs in the surface fundamental
  groups $ \pi_1(S_i) $. Arbitrarily choose one conjugacy representative $ h_j(0) $ from each pair, and let $ g_j(0) $ be the corresponding stable
  element of the HNN extension so that $ h_j(0)^{g_j(0)} $ represents the second element. We have chosen elements $ h_1(0),\ldots,h_m(0) $ and $ g_1(0),\ldots,g_m(0) $
  in $ G(0) $, and we can define $ h_1(t),\ldots,h_m(t) $ and $ g_1(t),\ldots,g_m(t) $ in $ G(t) $ to be their images under the canonical map $ G \to G(t) $.
  \end{subcons}

  \begin{subcons}[Deforming the skew-hexagon domains of peripheral groups]
  The proof of \cref{thm:funddom} gives us, for each $ j \in \{1,\ldots,m\} $ (resp. $ e \in E(T) $) and all $ t \in \mc{U} $, a pair of circles in $ \Pen(h_j(t)) $
  (resp. $\Pen(f_e(t))$) that are paired by $ h_j(t) $ (resp. $f_e(t)$). These circles were defined inductively and completely algebraically. Note that the inductive
  step for an element in some peripheral group $\Pi$ involves finding a certain circle that is orthogonal to $ \Delta(\Pi) $ at two points. Replacing `orthogonal
  to $ \Delta(\Pi)$' with `orthogonal to the circle through the three fixed points $ \Fix^+ A, \Fix^+ B, \Fix^+ C $' (where $ A$,$B$, and $C$ are primitive elements
  representing the three boundary components of $ \Delta(\Pi)/\Pi$), the circles obtained can be viewed as algebraically varying with parameters in $ X $ rather than
  just parameters on $ \mc{P}_\Lambda $. In addition, so long as the six circles for each $ \Pi \in \mathbold{\Pi} $ do not change their incidence structure (i.e.
  they continue to form a skew-hexagon in $ \P^3 $), the sum of the angles at the three points of intersction will remain constant at $ 2\pi $. This is because
  the sum of the angles is an analytic function (not algebraic, since it involves $ \arccos $) in the parameters $ X $ which is constant on the
  uncountable set $ \mc{U} $, so is constant in the entire connected component of its domain that contains $ \mc{U} $.
  \end{subcons}

  \begin{subcons}[Deforming the arcs paired by the HNN extensions]
  In addition to the circles just defined, the fundamental domains $ D_t $ contain certain circles paired by the stable elements $ g_j(t) $ (possibly twisted by $ h_j(t) $).
  We recall that there was, in the proof of \cref{thm:funddom}, a choice involved here: namely, we needed to choose a circle $ \hat{C}(h_j(t)) $ in the pencil of circles through
  the fixed points of $ h_j(t) $ (i.e.\ the orthogonal pencil to $ \Pen(h_j(t)) $). We will take the particular choice that $ \hat{C}(h_j(t)) $ is the Euclidean line through
  the fixed points (so $ C(h_j(t)) $ is the Euclidean segment joining the fixed points). This obviously moves algebraically with $ t \in \mc{U} $; we need to define a family
  of circles parameterised by $ t \in X $ which specialises to this choice on $ \mc{U} $ and which continues to satisfy the relevant holonomy conditions from the definition
  of a fundamental domain. Concretely, for $ t \in X $ and for each $ j $, we define (analogously to~\cite{ems21}) a twisted line that joins the two adjacent circles paired by $ h_j(t) $;
  the procedure is illustrated in \cref{fig:twisted_line}:

  \begin{figure}
    \labellist
    \small\hair 2pt
    \pinlabel {$F_{j',1}$} at 20 59
    \pinlabel {$F_{j,1}$} at 65 59
    \pinlabel {$F_{j',2}$} at 253 59
    \pinlabel {$F_{j,2}$} at 296 59
    \pinlabel {$w$} [b] at 80 88
    \pinlabel {$z$} [bl] at 114 72
    \pinlabel {$L$} [b] at 170 61
    \pinlabel {$\tilde{L}$} [tr] at 170 44
    \pinlabel {$w'$} [tr] at 222 15
    \endlabellist
    \centering
    \includegraphics[width=.6\textwidth]{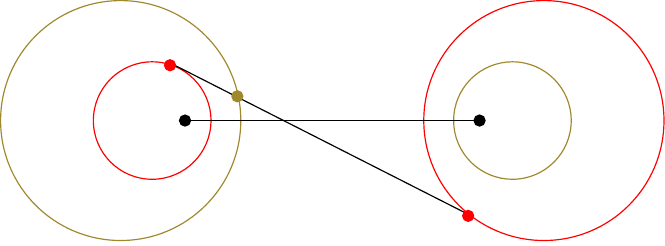}
    \caption{The `twisted' line that comes in two segments paired by $ g_j(t) h_j(t)^{\delta} $ and $g_j(t) h_j(t)^{\delta+1}$; the endpoints of $ \tilde{L} $
             on the circles $ F_{j,1} $ and $ F_{j,2} $ are correctly paired by $ h_j $. We also show $ F_{j',1} $ and $ F_{j',2} $, the inverse images under $ g_j $
             of the circles paired by $ g_j h_j g_j^{-1} $.\label{fig:twisted_line}}
  \end{figure}

  \begin{enumerate}[label={(\roman*)}]
    \item Let $ F_{j,1} $ and $ F_{j,2} $ be the two circles paired by $ h_j(t) $; let $ F'_{j,1} $ and $ F'_{j,2} $ be the circles paired by $ g_j(t) h_j(t) g_j^{-1}(t) $. Let $ \delta_j \in \Z $ be
          the coarse holonomy of $ g_j(t) $---by definition of $ \mc{U} $ this is a constant, independent of $ t $.
    \item Let $ L $ be the (Euclidean) line through the fixed points of $ h_j(t) $. The coordinates in $ \P^3 $ of this line can be determined algebraically in terms
          of the fixed points by \cref{lem:action_on_circles2}.
    \item Let $ w $ and $ w' $ be the intersection points of $ L $ with $ F_{j,1} $ and $ F_{j,2} $ respectively. This involves solving a quadratic equation in one variable.
    \item Compute the square root of the rotational part of $ h_j(t) $: change coordinates so that the fixed points of $ h_j(t) $ are $ 0 $ and $ \infty $, and $ h_j(t)(z) = \lambda^2 z $
          for $ \lambda \in \C $ (which will be an eigenvalue of $ h_j(t) $ as a $2\times 2$ matrix), then the desired map is $ \phi(z) = \lambda/\abs{\lambda} z$; this is $\R$-algebraic
          in the coefficients of $ h_j(t) $, modulo a single square root. Define $ \tilde{L}_j(t) $ to be the Euclidean segment through $ \phi^{-1}(w) $ and $ \phi(w') $.
    \item The circle $ h_j(t)^{\delta_j} ( g_j(t)^{-1}(F'_{j,1})) $ is algebraically determined by part (ii) \cref{lem:action_on_circles2}, i.e.\ it is the image under $ g_j(t)^{-1} $
          of the circle through the intersection points of $ F'_{j,1} $ with another circle in the same minimally bounded domain and with the boundary of the peripheral disc
          it is orthogonal to. Let $ z $ be the intersection of $ \tilde{L}_j(t) $ with $ h_j(t)^{\delta_j} ( g_j(t)^{-1}(F'_{j,1})) $.
  \end{enumerate}
  \end{subcons}

  \begin{subcons}[Synthesis]
  The formal curvilinear polygon with side-pairing structure $ D_t $ for all $ t \in X $ is now defined. Its sides and side-pairings are:
  \begin{itemize}
    \item The circles paired by the elements $ g_j(t) $ and $ f_e(t) $, defined by a simple extension of defining polynomials from $ \mc{U} $ to $X$.
    \item For each boundary component of $ S^* $, the segment $ [z,w'] \subset \tilde{L}_j(t) $ and its image under $ g_j(t) h_j(t)^{\delta_j} $, paired by the latter.
    \item For each boundary component of $ S^* $, the segment $ [w,z] \subset \tilde{L}_j(t) $ and its image under $ g_j(t) h_j(t)^{\delta_j+1} $, paired by the latter. (This will be
          a segment of a circle tangent to $g_j(t) h_j(t)^{\delta_j}[z,w'] $.)
  \end{itemize}
  \end{subcons}
\end{cons}

The regions defined by $ D_t $ are not necessarily fundamental domains for $ G(t) $ when $ t $ is sufficiently far from $ \mc{U} $, so
the construction alone does not certify discreteness.

\begin{thm}[\cref{mainthm2}]\label{thm:mainthm}
  Let $ G \leq \PSL(2,\C) $ be a convex cofinite manifold group which admits a $\Lambda$-circle chain for some maximal rational lamination $ \Lambda $.
  Let $ G(t) $ be a family of representations in $ \Hom(G,\PSL(2,\C)) $ that depends algebraically on $ t \in X $ for some parameter space $ X $,
  so that $ G(0) = G $. There exists an open, nontrivial, full-dimensional semi-algebraic neighbourhood $ \mc{Y} $ of $ G $ in $ \QH(G) \subset X $. Further, there is a countable
  open cover of $ \QH(G) $ consisting of sets of this form, indexed by the product of the set of maximal rational laminations with $ \Z^N$ for some large $N$ depending on $G$.
\end{thm}
\begin{proof}
  The set $ \mc{Y} $ will simply be a semi-algebraic set of points $ t \in X $ such that \cref{cons:twisting} `works'. It is clear that this set will be
  full-dimensional and nontrivial since it contains an open ball (with respect to the Teichm\"uller topology) around points close to $ G $ on the $\Lambda$-pleating variety
  (the domain in \cref{cons:twisting} is combinatorially stable under small movements of $t$ in any direction in $ \QH(G) $).

  We can form a fundamental domain for the deformed element so long as (i) all the peripheral groups are still discrete, and (ii) the domains of the individual
  peripheral groups have not started to overlap. All the circles defined in \cref{cons:twisting} are algebraically defined in terms of $ t $ and
  the initial choice of $ G $, modulo extraction of some square roots; these root extractions can be replaced by dummy variables, as we did in \cref{ex:schottky_85}, to keep the problem algebraic.
  The necessary conditions are:
  \begin{enumerate}
    \item \textit{Local conditions}. For each peripheral group $ \Pi_i(t)$ in the circle chain:
      \begin{enumerate}
        \item[L1.] The three pairs of incident circles in the domain for $ \Pi_i(t) $ in $ \Delta_i(t) $ must be mutually disjoint;
        \item[L2.] Each line $ \tilde{L}_{j}(t) $ must meet the two incident circles (the paired circles at
              its ends) at an angle in $ (0,\pi/2) $, which is an algebraic condition using the quadratic form $ \rad^2 $~\cite[Vol. II,~Proposition 20.4.2]{berger},
              and must not hit any other parameterised circle between these two incidence points.
      \end{enumerate}
    \item \textit{Global conditions}.
      \begin{enumerate}
        \item[G1.] For all $ i,j $, if the closures of the curvilinear domains for $ \Pi_i(0) $ and $ \Pi_j(0) $ are disjoint,
              then the closures of the curvilinear domains for $ \Pi_i(t) $ and $ \Pi_j(t) $ must be disjoint; that is, the total
              grafted fundamental domain must be without self-overlap. By \cref{lem:circle_intersection}, this amounts to a
              semi-algebraic condition.
        \item[G2.] Additional inequalities coming from \cite[Theorem~5.21]{basu} must be added to cut out the connected component of the semi-algebraic set
              that contains $ G $.
      \end{enumerate}
  \end{enumerate}
  If $ \mc{Y} $ is the semi-algebraic subset cut out by these inequalities and $ t \in \mc{Y} $, then the formal curvilinear polygon
  produced by \cref{cons:twisting} is an actual non-empty subset, locally cut out by circles paired by the putative side-pairing elements
  in the correct combinatorial arrangement, with the correct angle sums; in particular it defines a fundamental domain for $ G(t) $ as
  desired, varying algebraically with $ t $. There are two subtle points connected with the two global conditions.

  First, \textit{a priori} just having a fundamental domain for the conformal action of a subgroup of $ \PSL(2,\C) $ on the complex plane which satisfies
  the angle-sum conditions around vertices is not enough to guarantee discreteness: for example, the group could be a cone-manifold holonomy group with a
  cone angle along a closed loop that is not a sub-multiple of $\pi$, so the cone singularity is not visible from the conformal structure on the boundary (c.f.\ \cref{ex:cone_mfd}).
  It is here that (G2), i.e. that we restrict to the connected component containing the known discrete and geometrically finite base-point
  $ G $, is necessary: this rules out such degeneracies since, by the standard theorems about the boundary of $ \QH(G) $, any path which starts at $ G $ and ends at an indiscrete
  group must pass through a point $ G(t) \in \partial \QH(G) $ and for every such point either a boundary-parallel element of $ G(t) $ is parabolic that was
  not parabolic in $ G $ (i.e.\ $ G(t) $ is a cusp point) or a component of $ \Omega(G) $ vanishes in $ \Omega(G(t)) $ (i.e.\ $G(t) $ has a geometrically
  infinite end); both of these are detected by the conformal structures of $ G(t) $ and so the conditions we have given are sufficient to ensure that $G(t)$
  is in fact discrete.

  Second, in the usual formulations of polyhedron theorems it is necessary for translates of the fundamental domain under all elements of the group
  to be pairwise disjoint, while in (G1) we only require the domain to not overlap with itself. This is in fact sufficient, by the following short argument:
  if the polygon does not self-overlap and if the angle-sums around vertices add up to $ 2\pi $, then the side pairings define a complex structure on the
  correct union of surfaces, and the holonomy groups of these complex structures are the groups generated by the corresponding peripheral subgroups. By the argument in
  the previous paragraph again, any parabolics or arbitrarily short elements appearing in the global group $ G(t) $ must be visible in the pinching or vanishing of the
  complex structure of one of these surfaces, i.e.\ in the group generated by the peripheral subgroups $ \Pi_i(t) $, so until the point of overlap (where we no longer
  have a well-defined conformal surface obtained from the side pairings)
  we still have a correct fundamental domain so long as the local conditions (L1) and (L2) hold.

  From these inequalities, we obtain a countable family of semi-algebraic sets covering $ \mc{P}_\Lambda $, one for each covering set $ \mc{U}_j $ as explained in \cref{thm:funddom},
  and the union of these countably many sets for all $ \Lambda $ fill $ \QH(G) $ since pleating rays for maximal rational laminations are dense in $ \QH(G) $
  (see \cref{obs:cut_up_flat}).
\end{proof}

\begin{rem}
  Markowitz~\cite{markowitz24} has given an algorithm for discreteness in $ \SL(2,\R) $ which involves the recognition
  of peripheral structures in $ \H^2 $: the peripheral discs are arcs on $ \Sph^1 $, and a group is indiscrete if parts of the Cayley graph which `should'
  lead to the endpoints of such a peripheral structure (`leftmost paths') overlap in the interior of $ \H^2 $. The obstruction to extending this algorithm to $ \H^3 $
  is that peripheral structures can collapse without such overlaps occuring detectably. \Cref{thm:mainthm} is based on a similar idea: it starts with groups which are
  known to not have any overlaps (i.e.\ they have provably non-empty peripheral structures) and then deforms them slightly so that these structures do not self-intersect.
\end{rem}

\sloppy
\printbibliography

@MISC{agol04,
  AUTHOR = {Agol, Ian},
  DATE = {2004},
  EPRINT = {math/0405568},
  EPRINTCLASS = {math.GT},
  EPRINTTYPE = {arXiv},
  TITLE = {Tameness of hyperbolic $3$-manifolds},
}

@BOOK{akiyoshi,
  AUTHOR = {Akiyoshi, Hirotaka and Sakuma, Makoto and Wada, Masaaki and Yamashita, Yasushi},
  PUBLISHER = {Springer},
  DATE = {2007},
  DOI = {10.1007/978-3-540-71807-9},
  ISBN = {9783540718079},
  NUMBER = {1909},
  SERIES = {Lecture Notes in Math.},
  TITLE = {Punctured torus groups and $2$-bridge knot groups I},
}

@BOOK{aschenbrenner,
  AUTHOR = {Aschenbrenner, Matthias and Friedl, Stefan and Wilton, Henry},
  PUBLISHER = {European Mathematical Society},
  DATE = {2015},
  ISBN = {9783037191545},
  SERIES = {EMS Ser. Lect. Math.},
  TITLE = {$3$-manifold groups},
}

@BOOK{aurenhammer,
  AUTHOR = {Aurenhammer, Franz and Klein, Rolf and Lee, Ter-Tsai},
  PUBLISHER = {World Scientific},
  DATE = {2013},
  ISBN = {9789814447638},
  TITLE = {Voronoi diagrams and Delaunay triangulations},
}

@ARTICLE{baba23,
  AUTHOR = {Baba, Shinpei and Ohshika, Ken'ichi},
  DATE = {2023},
  DOI = {10.1093/imrn/rnac357},
  JOURNALTITLE = {Int. Math. Res. Not. IMRN},
  PAGES = {16674--16707},
  TITLE = {Realisation of bending measured laminations by Kleinian surface groups},
}

@BOOK{basu,
  AUTHOR = {Basu, Saugata and Pollack, Richard and Roy, Marie-Françoise},
  PUBLISHER = {Springer},
  DATE = {2006},
  ISBN = {9783540330981},
  NUMBER = {10},
  SERIES = {Algorithms Comput. Math.},
  TITLE = {Algorithms in real algebraic geometry},
}

@BOOK{beardon,
  AUTHOR = {Beardon, Alan F.},
  PUBLISHER = {Springer-Verlag},
  DATE = {1983},
  ISBN = {0387907882},
  NUMBER = {91},
  SERIES = {Grad. Texts in Math.},
  TITLE = {The geometry of discrete groups},
}

@MISC{benzvi22,
  AUTHOR = {Ben-Zvi, Michael and Lou, Jiayi and Walsh, Genevieve S.},
  DATE = {2022},
  EPRINT = {2201.12443},
  EPRINTCLASS = {math.GR},
  EPRINTTYPE = {arXiv},
  TITLE = {Hyperbolic boundaries vs. hyperbolic groups},
}

@BOOK{berger,
  AUTHOR = {Berger, Marcel},
  TRANSLATOR = {Cole, M. and Levy, S.},
  PUBLISHER = {Springer},
  DATE = {1987},
  EDITION = {Corrected fourth printing},
  NOTE = {Two volumes},
  SERIES = {Universitext},
  TITLE = {Geometry},
}

@ARTICLE{bonahonotal04,
  AUTHOR = {Bonahon, Francis and Otal, Jean-Pierre},
  DATE = {2004},
  DOI = {10.4007/annals.2004.160.1013},
  JOURNALTITLE = {Ann. Math. (2)},
  PAGES = {1013--1055},
  TITLE = {Laminations measuré{}es de plissage des varié{}té{}s hyperboliques de dimension 3},
  VOLUME = {160},
}

@ARTICLE{bowditch01,
  AUTHOR = {Bowditch, Brian H.},
  DATE = {2001},
  DOI = {10.1090/S0002-9947-01-02835-5},
  JOURNALTITLE = {Trans. Amer. Math. Soc.},
  PAGES = {4057--4082},
  TITLE = {Peripheral splittings of groups},
  VOLUME = {353},
}

@ARTICLE{brock12,
  AUTHOR = {Brock, Jeffrey F. and Canary, Richard D. and Minsky, Yair N.},
  DATE = {2012},
  DOI = {10.4007/annals.2012.176.1.1},
  EPRINT = {math/0412006},
  EPRINTCLASS = {math.GT},
  EPRINTTYPE = {arXiv},
  JOURNALTITLE = {Ann. Math. (2)},
  PAGES = {1--149},
  TITLE = {The classification of Kleinian surface groups, II: The ending lamination conjecture},
  VOLUME = {176},
}

@ARTICLE{bromberg11,
  AUTHOR = {Bromberg, Kenneth},
  DATE = {2011},
  DOI = {10.1215/00127094-2010-215},
  JOURNALTITLE = {Duke Math. J.},
  PAGES = {387--427},
  TITLE = {The space of Kleinian punctured torus groups is not locally connected},
  VOLUME = {156},
}

@ARTICLE{calegari06,
  AUTHOR = {Calegari, Danny and Gabai, David},
  DATE = {2006},
  DOI = {10.1090/S0894-0347-05-00513-8},
  EPRINT = {math/0407161},
  EPRINTCLASS = {math.GT},
  EPRINTTYPE = {arXiv},
  JOURNALTITLE = {J. Amer. Math. Soc.},
  PAGES = {385--446},
  TITLE = {Shrinkwrapping and the taming of hyperbolic $3$-manifolds},
  VOLUME = {19},
}

@ARTICLE{choi06,
  AUTHOR = {Choi, Young-Eun and Series, Caroline},
  DATE = {2006},
  DOI = {10.4310/jdg/1146680513},
  JOURNALTITLE = {J. Differential Geom.},
  PAGES = {75--117},
  TITLE = {Lengths are coordinates for convex structures},
  VOLUME = {73},
}

@ARTICLE{collins91,
  AUTHOR = {Collins, George E. and Hong, Hoon},
  DATE = {1991},
  DOI = {10.1016/S0747-7171(08)80152-6},
  JOURNALTITLE = {J. Symbolic Comput.},
  PAGES = {299--328},
  TITLE = {Partial Cylindrical Algebraic Decomposition for quantifier elimination},
  VOLUME = {12},
}

@BOOK{dicks,
  AUTHOR = {Dicks, Warren and Dunwoody, M.J.},
  PUBLISHER = {Cambridge University Press},
  DATE = {1989},
  ISBN = {9780521230339},
  NUMBER = {17},
  SERIES = {Cambridge Stud. Adv. Math.},
  TITLE = {Groups acting on graphs},
}

@MISC{dular25,
  AUTHOR = {Dular, Bruno},
  DATE = {2025},
  EPRINT = {2504.19891},
  EPRINTCLASS = {math.GT},
  EPRINTTYPE = {arXiv},
  TITLE = {Bending parameterization of one-sided degenerated Kleinian surface groups},
}

@MISC{dular24,
  AUTHOR = {Dular, Bruno and Schlenker, Jean-Marc},
  DATE = {2024},
  EPRINT = {2403.10090},
  EPRINTCLASS = {math.GT},
  EPRINTTYPE = {arXiv},
  TITLE = {Convex co-compact hyperbolic manifolds are determined by their pleating lamination},
}

@MISC{elzenaar24c,
  AUTHOR = {Elzenaar, Alex},
  DATE = {2024},
  EPRINT = {2411.17940},
  EPRINTCLASS = {math.GT},
  EPRINTTYPE = {arXiv},
  TITLE = {Changing topological type of compression bodies through cone manifolds},
}

@MISC{elzenaar25h,
  AUTHOR = {Elzenaar, Alex},
  DATE = {2025},
  EPRINT = {2503.13829},
  EPRINTCLASS = {math.GT},
  EPRINTTYPE = {arXiv},
  TITLE = {From disc patterns in the plane to character varieties of knot groups},
}

@MISC{ems24bd,
  AUTHOR = {Elzenaar, Alex and Gong, Jianhua and Martin, Gaven J. and Schillewaert, Jeroen},
  DATE = {2024},
  EPRINT = {2405.15970},
  EPRINTCLASS = {math.CV},
  EPRINTTYPE = {arXiv},
  TITLE = {Bounding deformation spaces of Kleinian groups with two generators},
}

@ARTICLE{ems21,
  AUTHOR = {Elzenaar, Alex and Martin, Gaven J. and Schillewaert, Jeroen},
  DATE = {2023},
  DOI = {10.1016/j.exmath.2022.12.002},
  EPRINT = {2111.03230},
  EPRINTCLASS = {math.GT},
  EPRINTTYPE = {arXiv},
  ISSUE = {1},
  JOURNALTITLE = {Expo. Math.},
  PAGES = {20--54},
  TITLE = {Approximations of the Riley slice},
  VOLUME = {41},
}

@ARTICLE{epstein88,
  AUTHOR = {Epstein, D. B. A. and Penner, R. C.},
  DATE = {1988},
  DOI = {10.4310/jdg/1214441650},
  JOURNALTITLE = {J. Differential Geom.},
  PAGES = {67--80},
  TITLE = {Euclidean decompositions of noncompact hyperbolic manifolds},
  VOLUME = {27},
}

@BOOK{ford,
  AUTHOR = {Ford, Lester R.},
  PUBLISHER = {Chelsea Publishing Company},
  DATE = {1951},
  EDITION = {2},
  TITLE = {Automorphic forms},
}

@ARTICLE{furokawa15,
  AUTHOR = {Furokawa, Mikio},
  DATE = {2015},
  DOI = {10.1016/j.topol.2015.05.017},
  JOURNALTITLE = {Topology Appl.},
  PAGES = {431--437},
  TITLE = {Ford domains of fuchsian once-punctured Klein bottle groups},
  VOLUME = {196},
}

@INBOOK{gilman23,
  AUTHOR = {Gilman, Jane},
  EDITOR = {Detinko, Alla and Kapovich, Michael and Kontorovich, Alex and Sarnak, Peter and Schwartz, Richard},
  BOOKTITLE = {Computational aspects of discrete subgroups of Lie groups},
  DATE = {2023},
  DOI = {10.1090/conm/783/15733},
  NUMBER = {783},
  PAGES = {57--68},
  SERIES = {Contemp. Math.},
  TITLE = {Computability models: algebraic, topological and geometric algorithms},
}

@ARTICLE{haissinsky16,
  AUTHOR = {Haïssinsky, Peter and Paoluzzi, Luisa and Walsh, Genevieve},
  DATE = {2016},
  DOI = {10.1215/ijm/1498032035},
  JOURNALTITLE = {Illinois J. Math.},
  PAGES = {353--364},
  TITLE = {Boundaries of Kleinian groups},
  VOLUME = {60},
}

@BOOK{hertrichjeromin,
  AUTHOR = {Hertrich-Jeromin, Udo},
  PUBLISHER = {Cambridge University Press},
  DATE = {2003},
  ISBN = {0521535697},
  NUMBER = {300},
  SERIES = {London Math. Soc. Lecture Note Ser.},
  TITLE = {Introduction to Möbius differential geometry},
}

@SOFTWARE{qepcad,
  AUTHOR = {Hong, Hoon and Brown, Christopher W. and others},
  URL = {https://www.usna.edu/Users/cs/wcbrown/qepcad/B/QEPCAD.html},
  DATE = {2018},
  HOWPUBLISHED = {Software},
  TITLE = {\texttt{QEPCAD}: Quantifier Elimination by Partial Cylindrical Algebraic Decomposition},
}

@ARTICLE{johansson00,
  AUTHOR = {Johansson, Stefan},
  DATE = {2000},
  DOI = {10.1090/S0025-5718-99-01167-9},
  JOURNALTITLE = {Math. Comput.},
  PAGES = {339--349},
  TITLE = {On fundamental domains of arithmetic Fuchsian groups},
  VOLUME = {69},
}

@INPROCEEDINGS{jorgensen03,
  AUTHOR = {Jørgensen, Troels},
  EDITOR = {Komori, Y. and Markovic, V. and Series, C.},
  PUBLISHER = {Cambridge University Press},
  BOOKTITLE = {Kleinian groups and hyperbolic $3$-manifolds},
  DATE = {2003},
  DOI = {10.1017/CBO9780511542817.010},
  NUMBER = {299},
  PAGES = {183--297},
  SERIES = {London Math. Soc. Lecture Note Ser.},
  TITLE = {On pairs of once-punctured tori},
}

@ARTICLE{kabaya11,
  AUTHOR = {Kabaya, Yuichi},
  DATE = {2014},
  DOI = {10.1007/s10711-013-9866-x},
  EPRINT = {1110.6674},
  EPRINTCLASS = {math.GT},
  EPRINTTYPE = {arXiv},
  JOURNALTITLE = {Geom. Dedicata},
  PAGES = {9--62},
  TITLE = {Parametrization of $\PSL(2,\C)$-representations of surface groups},
  VOLUME = {170},
}

@ARTICLE{kapovich16,
  AUTHOR = {Kapovich, Michael},
  DATE = {2016},
  DOI = {10.1142/S0218196716500193},
  JOURNALTITLE = {Internat. J. Algebra Comput.},
  PAGES = {467--472},
  TITLE = {Discreteness is undecidable},
  VOLUME = {26},
}

@INPROCEEDINGS{kapovich23discreteness,
  AUTHOR = {Kapovich, Michael},
  EDITOR = {Detinko, Alla and Kapovich, Michael and Kontorovich, Alex and Sarnak, Peter and Schwartz, Richard},
  BOOKTITLE = {Computational aspects of discrete subgroups of Lie groups},
  DATE = {2023},
  DOI = {10.1090/conm/783/15736},
  NUMBER = {783},
  PAGES = {87--112},
  SERIES = {Contemp. Math.},
  TITLE = {Geometric algorithms for discreteness and faithfulness},
}

@ARTICLE{keen73,
  AUTHOR = {Keen, Linda},
  DATE = {1973},
  DOI = {10.2307/2038632},
  JOURNALTITLE = {Proc. Amer. Math. Soc.},
  PAGES = {60--62},
  TITLE = {A correction to ``On Fricke moduli''},
  VOLUME = {40},
}

@INPROCEEDINGS{keen71,
  AUTHOR = {Keen, Linda},
  EDITOR = {Ahlfors, Lars V. and Bers, Lipman and Farkas, Hershel M. and Gunning, Robert C. and Kra, Irwin and Rauch, Harry E.},
  PUBLISHER = {Princeton University Press},
  BOOKSUBTITLE = {Proceedings of the 1969 Stony Brook Conference},
  BOOKTITLE = {Advances in the theory of Riemann surfaces},
  DATE = {1971},
  DOI = {10.1515/9781400822492-015},
  NUMBER = {66},
  PAGES = {205--224},
  SERIES = {Ann. of Math. Stud.},
  TITLE = {On Fricke moduli},
}

@ARTICLE{keen93,
  AUTHOR = {Keen, Linda and Series, Caroline},
  DATE = {1993},
  DOI = {10.1016/0040-9383(93)90048-z},
  JOURNALTITLE = {Topology},
  PAGES = {719--749},
  TITLE = {Pleating coordinates for the Maskit embedding of the Teichmüller space of punctured tori},
  VOLUME = {32},
}

@ARTICLE{keen94,
  AUTHOR = {Keen, Linda and Series, Caroline},
  DATE = {1994},
  DOI = {10.1112/plms/s3-69.1.72},
  JOURNALTITLE = {Proc. Lond. Math. Soc. (3)},
  PAGES = {72--90},
  TITLE = {The Riley slice of Schottky space},
  VOLUME = {69},
}

@article{keenseriesCCHB,
  author = {Linda Keen and Caroline Series},
  title = {Continuity of convex hull boundaries},
  journal = {Pacific J. Math.},
  volume = {168},
  number = {1},
  pages = {183--206},
  year = {1995},
  doi = {10.2140/pjm.1995.168.183}
}

@ARTICLE{komori97,
  AUTHOR = {Komori, Yohei},
  DATE = {1997},
  DOI = {10.2977/prims/1195145147},
  JOURNALTITLE = {Publ. Res. Inst. Math. Sci.},
  PAGES = {527--571},
  TITLE = {Semialgebraic description of Teichmüller space},
  VOLUME = {33},
}

@INBOOK{komoriseries98,
  AUTHOR = {Komori, Yohei and Series, Caroline},
  EDITOR = {Rivin, Igor and Rourke, Colin and Series, Caroline},
  PUBLISHER = {Mathematical Sciences Publishers},
  BOOKTITLE = {The Epstein birthday schrift},
  DATE = {1998},
  DOI = {10.2140/gtm.1998.1.303},
  EPRINT = {math/9810194},
  EPRINTCLASS = {math.GT},
  EPRINTTYPE = {arXiv},
  NUMBER = {1},
  PAGES = {303--316},
  SERIES = {Monogr. Geom. Topology},
  TITLE = {The Riley slice revisited},
}

@ARTICLE{kra90,
  AUTHOR = {Kra, Irwin},
  DATE = {1990},
  DOI = {10.2307/1990927},
  JOURNALTITLE = {J. Amer. Math. Soc.},
  PAGES = {499--578},
  TITLE = {Horocyclic coordinates for Riemann surfaces and moduli spaces, I. Teichmüller and Riemann spaces of Kleinian groups},
  VOLUME = {3},
}

@ARTICLE{lackenbypurcell13,
  AUTHOR = {Lackenby, Marc and Purcell, Jessica S.},
  DATE = {2014},
  DOI = {10.1080/10586458.2013.870503},
  EPRINT = {1302.3652},
  EPRINTCLASS = {math.GT},
  EPRINTTYPE = {arXiv},
  JOURNALTITLE = {Exp. Math.},
  PAGES = {218--240},
  TITLE = {Geodesics and compression bodies},
  VOLUME = {23},
}

@ARTICLE{maloni10,
  AUTHOR = {Maloni, Sara and Series, Caroline},
  DATE = {2010},
  DOI = {10.2140/agt.2010.10.1565},
  EPRINT = {1001.2515},
  EPRINTCLASS = {math.GN},
  EPRINTTYPE = {arXiv},
  JOURNALTITLE = {Algebr. Geom. Topol.},
  PAGES = {1565--1607},
  TITLE = {Top terms of polynomial traces in Kra's plumbing construction},
  VOLUME = {10},
}

@ARTICLE{manning02,
  AUTHOR = {Manning, Jason},
  DATE = {2002},
  DOI = {10.2140/gt.2002.6.1},
  EPRINT = {math/0102154},
  EPRINTCLASS = {math.GT},
  EPRINTTYPE = {arXiv},
  JOURNALTITLE = {Geom. Topol.},
  PAGES = {1--26},
  TITLE = {Algorithmic detection and description of hyperbolic structures on closed $3$-manifolds with solvable word problem},
  VOLUME = {6},
}

@ARTICLE{marden74g,
  AUTHOR = {Marden, Albert},
  DATE = {1974},
  DOI = {10.2307/1971059},
  JOURNALTITLE = {Ann. Math. (2)},
  PAGES = {383--462},
  TITLE = {The geometry of finitely generated Kleinian groups},
  VOLUME = {99},
}

@MISC{markowitz24,
  AUTHOR = {Markowitz, Ari},
  DATE = {2024},
  EPRINT = {2410.18315},
  EPRINTCLASS = {math.GR},
  EPRINTTYPE = {arXiv},
  TITLE = {Recognition and constructive membership for discrete subgroups of $\SL_2(\R)$},
}

@BOOK{maskit,
  AUTHOR = {Maskit, Bernard},
  PUBLISHER = {Springer-Verlag},
  DATE = {1987},
  ISBN = {978-3-642-61590-0},
  NUMBER = {287},
  SERIES = {Grundlehren Math. Wiss.},
  TITLE = {Kleinian groups},
}

@BOOK{matsuzakitaniguchi,
  AUTHOR = {Matsuzaki, Katsuhiko and Taniguchi, Masahiko},
  PUBLISHER = {Oxford University Press},
  DATE = {1998},
  ISBN = {0198500629},
  TITLE = {Hyperbolic manifolds and Kleinian groups},
}

@ARTICLE{minsky99,
  AUTHOR = {Minsky, Yair N.},
  DATE = {1999},
  DOI = {10.2307/120976},
  EPRINT = {math/9807001},
  EPRINTCLASS = {math.GT},
  EPRINTTYPE = {arXiv},
  JOURNALTITLE = {Ann. Math. (2)},
  PAGES = {559--626},
  TITLE = {The classification of punctured-torus groups},
  VOLUME = {171},
}

@ARTICLE{minsky10,
  AUTHOR = {Minsky, Yair N.},
  DATE = {2010},
  DOI = {10.4007/annals.2010.171.1},
  EPRINT = {math/0302208},
  EPRINTCLASS = {math.GT},
  EPRINTTYPE = {arXiv},
  JOURNALTITLE = {Ann. Math. (2)},
  PAGES = {1--107},
  TITLE = {The classification of Kleinian surface groups, I: Models and bounds},
  VOLUME = {171},
}

@ARTICLE{miyachi03,
  AUTHOR = {Miyachi, Hideki},
  DATE = {2003},
  DOI = {10.1090/S1088-4173-03-00065-1},
  JOURNALTITLE = {Conform. Geom. Dyn.},
  PAGES = {103--151},
  TITLE = {Cusps in complex boundaries of one-dimensional Teichmüller space},
  VOLUME = {7},
}

@ARTICLE{namazi12,
  AUTHOR = {Namazi, Hossein and Souto, Juan},
  DATE = {2012},
  DOI = {10.1007/s11511-012-0088-0},
  JOURNALTITLE = {Acta Math.},
  PAGES = {323--395},
  TITLE = {Non-realizability and ending laminations: proof of the density conjecture},
  VOLUME = {209},
}

@ARTICLE{ohshika20,
  AUTHOR = {Ohshika, Ken'ichi},
  DATE = {2020},
  DOI = {10.5802/afst.1647},
  EPRINT = {1010.0070v4},
  EPRINTCLASS = {math.GT},
  EPRINTTYPE = {arXiv},
  JOURNALTITLE = {Ann. Fac. Sci. Toulouse Math. (6)},
  PAGES = {805--895},
  TITLE = {Divergence, exotic convergence and self-bumping in quasi-Fuchsian spaces},
  VOLUME = {29},
}

@ARTICLE{ohshika11,
  AUTHOR = {Ohshika, Ken'ichi},
  DATE = {2011},
  DOI = {10.2140/gt.2011.15.827},
  EPRINT = {math/0504546},
  EPRINTCLASS = {math.GT},
  EPRINTTYPE = {arXiv},
  JOURNALTITLE = {Geom. Topol.},
  PAGES = {827--890},
  TITLE = {Realising end invariants by limits of minimally parabolic, geometrically finite groups},
  VOLUME = {15},
}

@INPROCEEDINGS{ohshika10,
  AUTHOR = {Ohshika, Ken'ichi and Miyachi, Hideki},
  EDITOR = {Bonk, Mario and Gilman, Jane and Masur, Howard and Minsky, Yair and Wolf, Michael},
  BOOKTITLE = {In the tradition of Ahlfors-Bers, V},
  DATE = {2010},
  NUMBER = {510},
  PAGES = {249--306},
  SERIES = {Contemp. Math.},
  TITLE = {Uniform models for the closure of the Riley slice},
}

@ARTICLE{parkerseries95,
  AUTHOR = {Parker, John R. and Series, Caroline},
  DATE = {1995},
  DOI = {10.1007/BF02787788},
  JOURNALTITLE = {J. Anal. Math.},
  PAGES = {165--198},
  TITLE = {Bending formulae for convex hull boundaries},
  VOLUME = {67},
}

@ARTICLE{poincare83,
  AUTHOR = {Poincaré, Henri},
  DATE = {1883},
  DOI = {10.1007/BF02422441},
  JOURNALTITLE = {Acta Math.},
  PAGES = {49--92},
  TITLE = {Mémoire sur les groupes Kleinéens},
  VOLUME = {3},
}

@INPROCEEDINGS{riley82,
  AUTHOR = {Riley, Robert},
  EDITOR = {Brown, R. and Thickstun, T.L.},
  PUBLISHER = {Cambridge University Press},
  BOOKTITLE = {Low-dimensional topology},
  DATE = {1982},
  DOI = {10.1017/CBO9780511758935.009},
  EVENTDATE = {1979},
  EVENTTITLE = {Conference on Topology in Low Dimension},
  ISBN = {0-521-28146-6},
  NUMBER = {48},
  PAGES = {81--151},
  SERIES = {London Math. Soc. Lecture Note Ser.},
  TITLE = {Seven excellent knots},
  VENUE = {University College of North Wales, Bangor},
}

@INPROCEEDINGS{saito94f,
  AUTHOR = {Saito, Kyoji},
  EDITOR = {Schneps, Leila},
  PUBLISHER = {Cambridge University Press},
  BOOKTITLE = {The Grothendieck theory of dessins d’enfants},
  DATE = {1994},
  NUMBER = {200},
  PAGES = {255--288},
  SERIES = {London Math. Soc. Lecture Note Ser.},
  TITLE = {Algebraic representation of the Teichmüller spaces},
}

@INPROCEEDINGS{series06,
  AUTHOR = {Series, Caroline},
  EDITOR = {Minsky, Yair N. and Sakuma, Makoto and Series, Caroline},
  PUBLISHER = {Cambridge University Press},
  BOOKTITLE = {Spaces of Kleinian groups},
  DATE = {2006},
  DOI = {10.1017/CBO9781139106993.005},
  ISBN = {9780521617970},
  NUMBER = {329},
  PAGES = {75--89},
  SERIES = {London Math. Soc. Lecture Note Ser.},
  TITLE = {Thurston's bending measure conjecture for once punctured torus groups},
}

@ARTICLE{thurston82,
  AUTHOR = {Thurston, William P.},
  DATE = {1982},
  DOI = {10.1090/s0273-0979-1982-15003-0},
  JOURNALTITLE = {Bull. Amer. Math. Soc. (N.S.)},
  PAGES = {357--381},
  TITLE = {Three dimensional manifolds, Kleinian groups, and hyperbolic geometry},
  VOLUME = {6},
}

@ARTICLE{tukia85,
  AUTHOR = {Tukia, Pekka},
  URL = {http://www.numdam.org/item/PMIHES_1985__61__171_0/},
  DATE = {1985},
  JOURNALTITLE = {Inst. Hautes Études Sci. Publ. Math.},
  PAGES = {171--214},
  TITLE = {On isomorphisms of geometrically finite groups},
  VOLUME = {61},
}

\end{document}